\RequirePackage{fix-cm}
\documentclass{rsepublic}       
\Year{1999}
\Volume{129}
\Paper{A99199}
\Pagerange{\pageref{firstpage}--\pageref{lastpage}}
\usepackage{graphicx}
\usepackage{booktabs}
\usepackage[toc,page]{appendix}
\usepackage{enumerate}
\usepackage{mathrsfs}
\usepackage{fancyhdr}
\usepackage{latexsym, amsfonts, amssymb, amsmath}
\usepackage{amsopn, amstext, amscd,pifont}
\usepackage{amssymb,bm, cite}
\usepackage[all]{xy}
\usepackage{color}
\usepackage{graphicx,hyperref}
\usepackage{float}
\usepackage{listings}
\usepackage[T1]{fontenc}
\usepackage{yfonts}
\usepackage{tkz-euclide}
\usepackage{manfnt}
\usepackage{xcolor}
\usepackage{graphicx,pgf}
\usepackage{float}
\usepackage{caption}
\usepackage{subcaption}

\def\dd{{\mathrm{d}}}
\newtheorem{thm}{Theorem}[section]
\numberwithin{thm}{section}
\newtheorem{lem}[thm]{Lemma}
\newtheorem{pro}[thm]{Proposition}

\newtheorem{cor}[thm]{Corollary}
\newtheorem{de}[thm]{Definition}

\newtheorem{rem}[thm]{Remark}

\numberwithin{equation}{section}
\numberwithin{figure}{section}

\allowdisplaybreaks

\begin{document}

\title{Pointwise stability of reaction diffusion fronts}


\author{Yingwei Li}

\footnotetext{Indiana University, 831 East Third Street, Rawles Hall, Bloomington, Indiana 47405, U.S.A., YL37@umail.iu.edu}

\MSdates[`Received date']{`Accepted date'}
\label{firstpage}
\maketitle

\begin{abstract}
Using pointwise semigroup techniques, we establish sharp rates of decay in space and time of a perturbed reaction
diffusion front to its time-asymptotic limit.
This recovers results of Sattinger, Henry and others of time-exponential convergence in weighted $L^p$ and Sobolev norms, while capturing the new feature of spatial diffusion at Gaussian rate. Novel features of the argument are a pointwise Green function decomposition reconciling spectral decomposition and short-time Nash-Aronson estimates and an instantaneous tracking scheme similar to that used in the study of stability of viscous shock waves.
\end{abstract}

\section{Introduction}

In this paper, we revisit the problem of stability of reaction diffusion fronts treated by Sattinger, Henry, and others
\cite{Sa,He,TZ} by essentially ODE methods, from the viewpoint of the pointwise semigroup methods introduced in
\cite{ZH,MaZ} for the study of stability of viscous shock waves.
For simplicity, we treat the semilinear case, with Laplacian diffusion; however, our methods readily extend to
the general second-order quasilinear strictly parabolic case (see, e.g., \cite{ZH} in the shock wave case).

From the ODE perspective, a stationary front solution of a spatially homogeneous parabolic system in one dimension
is an equilibrium which, due to translation invariance of the underlying system, is embedded in a one-parameter family
of nearby equilibria given by its translates.
Assuming that the front is a transversal connection of the associated standing-wave ODE,
its linearized operator $L$ has a one-dimensional zero eigenspace precluding asymptotic stability.
However, under the assumption of a spectral gap,
i.e., assuming that the rest of the spectrum has strictly negative real part,
one may hope to establish asymptotic orbital stability with time-exponential convergence to an appropriate element of
the {\it family} of nearby fronts.

This is indeed what was shown (among a number of other things; see Remark \ref{satrmk} below)
by Sattinger \cite{Sa}, by what is essentially a stable manifold construction about the
(normally hyperbolic) curve of equilibria.
A particularly simple alternative argument
is sketched by Henry in \cite[Exercise 6, p. 108]{He}, based on a normal-form
reduction approximately decoupling the normal and tangential flows.
Another approach, as described, e.g., in \cite{TZ}, is to ``factor out'' the group symmetry of translation, working effectively
on the quotient space, and thereby reducing orbital to asymptotic stability.

\begin{rem}\label{satrmk}
The main point of \cite{Sa} was to treat spatially inhomogeneous systems
arising through the imposition of spatially-weighted norms on reaction-diffusion-convection equation;
for a thematically-related study in the finite-dimensional setting,
see the argument of Cronin \cite[p. 198 ]{Cr} for stability of time-periodic orbits, featuring a similar
stable-manifold construction on the nonautonomous system obtained by Floquet transformation.
\end{rem}

Each of these arguments is based ultimately on the standard spectral decomposition of the linearized solution operator
$S(t):=e^{Lt}$ into the projection $Pf:=\phi \langle \tilde \psi, f\rangle$ onto $ {\rm Kernel} \, L$, where $\phi$ and $\tilde{\psi}$
are right and left zero eigenfunctions, $\langle \tilde \psi, \phi\rangle=1$,
with $\langle \cdot, \cdot\rangle$ denoting $L^2$ inner product, plus a time-exponentially decaying
portion associated with the complementary eigenspace associated with the remaining stable spectra,
together with some form of Duhamel's formula/variation of constants.
That is, they all effectively approximate the Green function $G(x,t;y)$ of the
linearized equations with the kernel
\begin{equation}\label{projapprox}
k(x,y):-\phi(x)\tilde \psi(y)\sim e^{-\eta_\pm |x|-\nu_\pm |y|}
\end{equation}
as $|x|,|y|\to \infty$ of the translational projection $P$, for some $\eta_\pm,\nu_\pm>0$.

On the other hand, the Nash-Aronson bounds \cite{N,A} of standard short-time parabolic theory yield that the
Green function is bounded above and below by Gaussian distributions:
\begin{equation}\label{NA}
	C_1 t^{1/2}e^{-|x-y|^2/M_1 t}\leq |G(x,t;y)| \leq C_2 t^{1/2}e^{-|x-y|^2/M_2 t},
\end{equation}
where $M_j>0$ constant, $C_j>0$ are bounded above and below for any fixed time interval $0\leq t\leq T$.
Comparing the quadratic exponential decay of the Gaussian with the linear exponential decay of $k(\cdot, \cdot)$,
we see that \eqref{projapprox}, though optimal with respect to time does not give an accurate picture of the spatial
propagation of data via Gaussian diffusion.  Likewise, the bound \eqref{NA} gives no information about large-time
asymptotics.

The goal of the present analysis, as we now describe, is to reconcile these two points of view,
obtaining estimates on linear and nonlinear behavior that are optimal both in the large-$t$ and large-$x$ regimes;
that is, to reconcile ODE and PDE estimates to obtain sharp pointwise dynamics.

\medskip

Consider a stationary front solution $u(x,t)=\bar{u}(x), \lim_{z\to \pm \infty}\bar{u}(z)=u_{\pm}$
of a system of reaction diffusion equations
\begin{equation}
u_t=u_{xx}+f(u). \label{readiff}
\end{equation}
For simplicity, take $f\in C^\infty$ throughout the paper.
Obviously, $\bar{u}$ satisfies
\begin{equation}
\bar{u}_{xx}+f(\bar{u})=0. \label{readiffprofile}
\end{equation}

Linearizing \eqref{readiff} about $\bar{u}$, we have
\begin{equation}
v_t=Lv:=v_{xx}+Df(\bar{u})v. \label{readifflinearoperator}
\end{equation}

The homogeneous linearized equation
\begin{equation}
v_t-Lv=v_t-v_{xx}-Df(\bar{u})v=0, \quad v|_{t=0}=g \label{homogeneouslinearized}
\end{equation}
can be solved by
\begin{equation}
e^{Lt}g=\int_{-\infty}^{+\infty}G(x,t;y)g(y)dy
\end{equation}
where $G(x,t;y)=S(t)\delta_{y}(x)=e^{Lt}\delta_{y}(x)$ is the Green function for \eqref{homogeneouslinearized}.

Differentiating \eqref{readiffprofile}, we obtain the standard fact that the translational
mode $\bar u'$ is a zero eigenfunction of the linearized operator $L$.
Introduce the spectral stability condition:

{\it $(\mathcal{D})$: The operator $L$ has a simple eigenvalue at $\lambda=0$ (with eigenfunction $\bar{u}^{\prime}$),
with all other spectrums of $L$ satisfying $\mathrm{Re}\lambda <-\eta$ for some $\eta >0$.}

\medskip
\begin{rem}
This has the consequence that eigenvalues of $Df(u_{\pm})$ have strictly negative real parts, corresponding to the stability of the end states $u_{\pm}$ as equilibria of the reaction ODE, $\dot{u}=f(u)$, ignoring the effects of diffusion.
\end{rem}

Assuming $(\mathcal{D})$, we obtain in standard fashion that the limits $u_{\pm}$ as $x\to \pm \infty$ of $\bar u(x)$ are
hyperbolic rest points of \eqref{readiffprofile} and thus, by the Stable Manifold Theorem:

\begin{pro}
Under assumption $(\mathcal{D})$, there exists a constant $C>0$ such that
\begin{eqnarray}
|\bar{u}(x)-u_{\pm}|\leq Ce^{-\eta |x|}, \quad x\gtrless 0. \label{profileproperty}
\end{eqnarray}
\end{pro}

Our first result is a pointwise version of the standard approximation by \eqref{projapprox}:
\begin{pro} \label{Proposition1.2}
Under assumption $(\mathcal{D})$, the Green function can be decomposed as
\begin{equation}
G(x,t;y)=E(x,t;y)+\tilde{G}(x,t;y), \label{greenfunctiondecom}
\end{equation}
where
\begin{equation}
E(x,t;y)=\bar{u}^{\prime}(x)e(y,t),
\end{equation}
\begin{equation} \label{e(y,t)}
e(y,t)=\chi(t)\tilde{\psi}(y).
\end{equation}
Here, $\chi(t)$ is a $C^{\infty}$ cutoff function satisfying $\chi(t)\equiv 0$ for $0\leq t\leq 1$ and $\chi(t)\equiv 1$ for $t\geq 2$. 
$\tilde{\psi}(y)$ is a left eigenvector of the linearized solution operator $S(t)=e^{Lt}$.

For any $\eta_0$ satisfying $0<\eta_0<\min(\eta/4,\eta^{\prime})$, $\eta^{\prime}$ as defined in Proposition \ref{Proposition4.1}, there exist positive constants $C_0, C_1, C_2$ and $C$ for which the following inequalities hold:
\begin{equation} \label{tilde_G_estimate}
\begin{aligned}
|\tilde{G}(x,t;y)|&\leq C_1 t^{-\frac{1}{2}} e^{-\eta_0 t-\frac{|x-y|^2}{4C_0 t}}+C_2 e^{-\eta_0(t+|x-y|)},
\end{aligned}
\end{equation}
\begin{equation} \label{tilde_G_y_estimate}
\begin{aligned}
|\tilde{G}_y(x,t;y)|\leq C_1 t^{-1} e^{-\eta_0 t-\frac{|x-y|^2}{4C_0 t}}+C_2 e^{-\eta_0(t+|x-y|)},
\end{aligned}
\end{equation}
\begin{equation} \label{e(y,t)estimate}
\begin{aligned}
|e(y,t)|\leq Ce^{-\eta_0 |y|}, &\quad |e_y(y,t)|\leq Ce^{-\eta_0 |y|},\\
|e_t(y,t)|\leq Ce^{-\eta_0 (t+|y|)}, &\quad |e_{ty}(y,t)|\leq C e^{-\eta_0 (t+|y|)}.
\end{aligned}
\end{equation}
\end{pro}

Our second result refines \eqref{projapprox}, capturing
Gaussian spatial propagation of perturbations:
\begin{pro} \label{Proposition1.3}
Under assumption $(\mathcal{D})$, the Green function can be decomposed as
\begin{equation}
G(x,t;y)=F(x,t;y)+\tilde{H}(x,t;y), \label{green_func_decom_Impro}
\end{equation}
where
\begin{equation} \label{F(x,t;y)}
F(x,t;y)=\bar{u}^{\prime}(x)\tilde{e}(x,t;y),
\end{equation}
\begin{equation}
\tilde{e}(x,t;y)=\chi(t)\tilde{\psi}(y)\left(\mathrm{errfn}\left(\frac{x-y+t}{\sqrt{4t}}\right)
 -\mathrm{errfn}\left(\frac{x-y-t}{\sqrt{4t}}\right)\right).
\end{equation}

Here, $\chi(t)$ and $\tilde{\psi}(y)$ are the same as in ${\rm Proposition}$ \ref{Proposition1.2}.

For $0<\eta_0<\min(\eta/4,\eta^{\prime}, \frac{1}{16} )$, and any integer $k,m\geq 0$, there exist constants $M,C,C_m >0$ sufficiently large such that
\begin{equation} \label{tilde_H_estimate}
\begin{aligned}
|\tilde{H}(x,t;y)|\leq Ct^{-\frac{1}{2}}e^{-\eta_0 t-\frac{|x-y|^2}{M t}},
\end{aligned}
\end{equation}
\begin{equation} \label{tilde_H_y_estimate}
\begin{aligned}
|\tilde{H}_y(x,t;y)|\leq Ct^{-1}e^{-\eta_0 t-\frac{|x-y|^2}{Mt}},
\end{aligned}
\end{equation}
\begin{equation} \label{tilde_e(y,t)estimate}
\begin{aligned}
|\partial_{t}^{k}\partial_{x}^{m}\tilde{e}(x,t;y)|\leq& C_m e^{-\eta_0 |y|}
\left|\frac{e^{-\frac{(x-y+t)^2}{Mt}}}{\sqrt{t+1}}-\frac{e^{-\frac{(x-y-t)^2}{Mt}}}{\sqrt{t+1}}\right|.
\end{aligned}
\end{equation}
\end{pro}

From Proposition \ref{Proposition1.2}, we obtain the following theorem
recovering the $L^p$ results of \cite{Sa,He}. Let $L^{p_0}, L^{p}, L^{\infty}$ denote the space of functions with finite $\|\cdot\|_{L^{p_0}}, \|\cdot\|_{L^{p}}, \|\cdot\|_{L^{\infty}}$-norms in the space variable $x \in \mathbb{R}$, and $W_{t}^{1,\infty}$ denote the space of functions with finite $\|\cdot\|_{W^{1,\infty}}$-norm in the time variable $t\in (0,\infty)$.
\begin{thm}[Nonlinear Stability] \label{mainthm}
Assuming $(\mathcal{D})$,
for any $p_0\geq 1$, stationary solutions $\bar{u}(x)$ of \eqref{readiff} are nonlinearly stable in $L^{p_0}\cap L^{\infty}$ and nonlinearly orbitally asymptotically stable in $L^p$, $p\geq p_0$, with respect to initial perturbations $u_0$ that are sufficiently small in $L^{p_0}\cap L^{\infty}$. More precisely, there exist some $C>0$ and $\alpha \in W_{t}^{1,\infty}$, such that
\begin{eqnarray*}
|\tilde{u}(x,t)-\bar{u}(x-\alpha(t))|_{L^p(x)}&\leq&Ce^{-\eta_0 t}
|\tilde{u}-\bar{u}|_{L^{p_0}\cap L^{\infty}}|_{t=0},\\
|\dot{\alpha}(t)|&\leq&Ce^{-\eta_0 t}|\tilde{u}-\bar{u}|_{L^{p_0}\cap L^{\infty}}|_{t=0},\\
|\alpha(t)|&\leq&C|\tilde{u}-\bar{u}|_{L^{p_0}\cap L^{\infty}}|_{t=0},\\
|\tilde{u}-\bar{u}|_{L^{p}}(t)&\leq&C|\tilde{u}-\bar{u}|_{L^{p_0}\cap L^{\infty}}|_{t=0},
\end{eqnarray*}
for all $t\geq 0$, $p_0\leq p\leq\infty$, for solutions $\tilde{u}$ of \eqref{readiff} with
$|\tilde{u}-\bar{u}|_{L^{p_0}\cap L^{\infty}}|_{t=0}$ sufficiently small.
\end{thm}


From Proposition \ref{Proposition1.3}, we obtain the following improved theorem
describing at once both time-exponential decay and Gaussian spatial diffusion of perturbations:
\begin{thm}[Pointwise Nonlinear Stability] \label{mainthm2}
Assume $(\mathcal{D})$ and let $\tilde{u}(x,t)$ be a solution of \eqref{readiff}. There exist positive constants $E_0, M$ and a function $\tilde{\alpha}=\tilde{\alpha}(x,t)\in W_{t}^{1,\infty}((0,\infty);W_{x}^{k,p}(\mathbb{R}))$ for all $k\geq 0, 1\leq p\leq \infty$ such that if
$$
|u_0(x)|=|\tilde{u}(x,0)-\bar{u}(x)|\leq E_0e^{-\frac{|x|^2}{M}}
$$
then $\tilde{u}(x+\tilde{\alpha}(x,t),t)$, $\tilde{\alpha}(x,t)$ and its derivatives satisfy the following pointwise bounds:
\begin{eqnarray*}
|\tilde{u}(x+\tilde{\alpha}(x,t),t)-\bar{u}(x)|&\leq&
CE_0(1+t)^{-\frac{1}{2}}e^{-\frac{\eta_0}{2}t-\frac{|x|^2}{2M(1+t)}},\\
|\tilde{\alpha}(x,t)|&\leq&CE_0\left|\mathrm{errfn}\left(\frac{x+t}{\sqrt{Mt}}\right)
-\mathrm{errfn}\left(\frac{x-t}{\sqrt{Mt}}\right)\right|,\\
|\partial_t^k\partial_x^m\tilde{\alpha}(x,t)|&\leq& 
CE_0(1+t)^{-\frac{1}{2}}
	\left(e^{-\frac{|x+t|^2}{Mt}}+e^{-\frac{|x-t|^2}{Mt}}\right), 
	\quad k+m\geq 1,
\end{eqnarray*}
for some $\eta_0>0$ small enough and $C>0$ large enough.
\end{thm}

Note that, differently than in Theorem \ref{mainthm}, the phase shift $\tilde{\alpha}(x,t)$ in Theorem \ref{mainthm2} is allowed to vary in $x$ as well as $t$, paralleling the linear behavior of $\tilde{e}$.

From the bounds of Theorem \ref{mainthm}, the phase shift $\alpha(t)$ converges exponentially to a constant shift $\alpha_\infty$, yielding 
the standard result \cite{Sa,He} that, for any $1\leq p\leq \infty$, for small $L^p$ perturbations,  $\tilde u$ converges time-exponentially in $L^p$ 
to a translate $\bar u(\cdot -\alpha_\infty)$ of the background traveling wave.
Comparing to the description of $\tilde u$ in Theorem \ref{mainthm2}, we find therefore for small Gaussian perturbations
that $\tilde \alpha(x,t)$ converges pointwise uniformly time-exponentially to a constant value $\tilde\alpha_\infty=\alpha_\infty$
on an expanding cone $\mathcal{C}:=\{|x|\leq \theta t\}$, some $\theta>0$.
(Indeed, this could be obtained directly from the analysis as well, by the observation that $\tilde e(x,t;y)e^{-\eta_0 |y|}$ 
converges time-exponentially to a constant on $\mathcal{C}$.)
This gives the refined, pointwise picture of convergence to a constant shift on the parabolic domain of dependence of the perturbation
data, with falloff at Gaussian rate as $x\to \pm\infty$ toward the unperturbed traveling wave $\bar u$.

\subsection*{\bf{Discussion}}\label{disc}
Stability of reaction diffusion fronts has been much studied, by a variety of techniques.
Indeed, the bounds on $\tilde G$ stated in Proposition \ref{Proposition1.2} may be recognized as exactly what one expects for a sectorial
ordinary differential operator possessing an exponential dichotomy, or, equivalently, a spectral gap, and could be obtained by a number of different
(standard) methods.
They serve here, along with the argument for Theorem \ref{mainthm}, as a bridge linking such standard methods with the approach used to establish Proposition
\ref{Proposition1.3} and Theorem \ref{mainthm2}, which represent the novel aspects of this work.
To our knowledge, no such results have up to now been obtained, despite the long history of the subject, and the naturality of the question they answer: of
how PDE properties such as parabolicity intervene in the ODE-like asymptotic behavior of the solution.

The latter are obtained by suitable adaptations of the ``pointwise semigroup'' methods introduced in \cite{ZH,MaZ} for the treatment of the neutral case of
stability of viscous shock waves, for which the associated linearized operator does not possess a spectral gap; see \cite{Z,Li} for particularly accessible
accounts specializing to the scalar case.
The treatment here of the more standard case with spectral gap both illuminates the method, and shows that it can give new details even in the background of
strong, time-exponential decay.
In the pointwise semigroup method, the resolvent and solution operators are replaced by their kernels, which may then be studied separately in different
$(x,y,t)$ domains; this is particularly natural for questions as addressed here of behavior in specific asymptotic situations.

Note that, in the nonlinear part of the argument, we have generalized the approach of \cite{ZH,MaZ} by allowing phase shifts depending on $x$ as well as $t$.  This is similar to, and motivated by, techniques introduced for the study of stability of periodic traveling wave \cite{JZ,JNRZ}, where it is crucial for completion of a nonlinear stability argument.  Our results here show that, also in the traveling front or pulse case, where it is not needed in order to close a nonlinear stability argument, allowing $\alpha$ to depend on $x$ is a useful tool, that can be used to obtain additional details about behavior.

Note, finally, that the Green function bounds of Proposition \ref{Proposition1.3} both recover and extend the classical Nash-Aronson bounds, identifying
a "parabolic" regime $|x-y|\gg t$ on which they dominate behavior, including, but not limited to the classical bounded-time regime.

\medskip

\noindent {\bf Outline of Proof}:
We will give a motivation of the pointwise semigroup methods in section \ref{PointwiseSemigroup}. In section \ref{AsymptoticEigenvalueEq}, we will show that the operator $\mathbb{A}(x;\lambda)$ associated with the eigenvalue equation of the linearized operator $L$ of equation \eqref{readiff} is asymptotically constant in $x$ when $|x|\to \infty$, and we will give a construction of the resolvent kernel $G_{\lambda}$ in section \ref{ConstructionoftheResolventKernel}, then we use this construction to give bounds on the resolvent kernel in sections \ref{LowandBoundedFrequencyBounds} and \ref{HighFrequencyBounds}. After that, in sections \ref{PointwiseBoundsontheGreenFunction} and \ref{ImprovedPointwiseBoundsontheGreenFunction}, we use these bounds on the resolvent kernel to get the bounds on the Green function, proving Propositions \ref{Proposition1.2} and \ref{Proposition1.3}. With the Green function kernel estimates in hand, we can go on to prove some $L^p\to L^p$ estimates on the Green function $G(x,t;y)$ in section \ref{LinearGreenEst}. With all these preparations done, we define and prove estimates on the time dependent translate $\alpha(t)$ in section \ref{Section_Integral_rep} and \ref{LpNonlinearIterationandLpNonlinearStability}, proving Theorem \ref{mainthm}. In section \ref{IntegralRepresentationforHKandPointwiseIterationSchemes}, we define the time and space dependent translate $\tilde{\alpha}(x,t)$;
its estimates as well as estimates on the perturbation $v(x,t)$ are given in section \ref{HKNonlinearIteration} and section \ref{PointwiseNonlinearIterationandPointwiseBoundonthePerturbation}, proving Theorem \ref{mainthm2}.

\medskip

\noindent{\bf Notation:} 
We use $C$ to denote a universal constant that may change from line to line but is independent of 
parameters, initial data, space, or time. We use the notation $f=\mathcal{O}(g)$ to mean that $|f|\leq C|g|$. 

\section{Pointwise Semigroup Methods} \label{PointwiseSemigroup}

We study the \textit{resolvent kernel} $G_{\lambda}(x,y)$, defined formally by
$$
G_{\lambda}(\cdot,y):=(L-\lambda I)^{-1}\delta_y,
$$
or equivalently
$
(L-\lambda I)^{-1}f(x)=\int G_{\lambda}(x,y)f(y)dy,
$
that is, the elliptic Green function associated with $(L-\lambda I)$. At an isolated eigenvalue $\lambda_0$ of $L$, the spectral projection operator can be defined by
\begin{equation*}
\mathcal{P}_{\lambda_0}=\mathrm{Res}_{\lambda_0}(L-\lambda I)^{-1}.
\end{equation*}

The operator $L$ is sectorial, so we have the spectral resolution formula,
\begin{equation} \label{spectralresolution}
e^{Lt}=\frac{1}{2\pi i}\int_{\Gamma}e^{\lambda t}(L-\lambda I)^{-1}d\lambda,
\end{equation}
for the solution operator $e^{Lt}$ of $v_t=Lv;\quad v(0)=v_0$, where $\Gamma$ is the boundary of an appropriate sector
$\{\lambda:\mathrm{Re}\lambda<\theta_1-\theta_2|\mathrm{Im}\lambda|\}$ containing the spectrum of $L$, $\theta_1,\theta_2>0$ are constants. We have assumed in assumption $(\mathcal{D})$ that $L$ has an isolated simple eigenvalue at $\lambda=0$, the rest of the spectrum is separated by a \textit{positive spectral gap} $\eta>0$, $\sigma(L)\setminus \{0\} \subset \{\lambda: \mathrm{Re}\lambda\leq -\eta \}$.
Defining $\tilde{\Gamma}$ as the boundary of the set
$\Omega:=\{\lambda:\mathrm{Re}\lambda<\theta_1-\theta_2|\mathrm{Im}\lambda|\}\cap\{\lambda: \mathrm{Re}\lambda\leq -\eta/2\}$, we have by \eqref{spectralresolution}, together with Cauchy's theorem, that
\begin{eqnarray*}
e^{Lt}=\mathrm{Res}_0 e^{\lambda t}(L-\lambda I)^{-1}
+\frac{1}{2\pi i}\int_{\tilde{\Gamma}}e^{\lambda t}(L-\lambda I)^{-1}d\lambda.
\end{eqnarray*}
Applying both sides of the above equation to $\delta_y(x)$ gives
\begin{eqnarray*}
G(x,t;y)
&=&\bar{u}^{\prime}(x)\left(\frac{1}{2\pi i}\int_{\partial B(0,\epsilon)}\frac{e^{\lambda t}}{\lambda}d\lambda\right)\tilde{\psi}(y)
   +\frac{1}{2\pi i}\int_{\tilde{\Gamma}}e^{\lambda t}{G}_{\lambda}(x,y)d\lambda\\
&=&\bar{u}^{\prime}(x)\tilde{\psi}(y)+\frac{1}{2\pi i}\int_{\tilde{\Gamma}}e^{\lambda t}{G}_{\lambda}(x,y)d\lambda\\
&=&\bar{u}^{\prime}(x)\chi(t)\tilde{\psi}(y)+\bar{u}^{\prime}(x)(1-\chi(t))\tilde{\psi}(y)
   +\frac{1}{2\pi i}\int_{\tilde{\Gamma}}e^{\lambda t}{G}_{\lambda}(x,y)d\lambda\\
&=&\bar{u}^{\prime}(x)e(y,t)+\tilde{G}(x,t;y),
\end{eqnarray*}
where we have defined $e(y,t)=\chi(t)\tilde{\psi}(y)$, and
	$$
\tilde{G}(x,t;y)=\bar{u}^{\prime}(x)(1-\chi(t))\tilde{\psi}(y)
   +\frac{1}{2\pi i}\int_{\tilde{\Gamma}}e^{\lambda t}{G}_{\lambda}(x,y)d\lambda.
   $$
The cutoff function $\chi(t)\in C^{\infty}(\mathbb{R}^{+})$ is identically $0$ for $0\leq t\leq 1$, and identically $1$ for $t\geq 2$.

\section{The Asymptotic Eigenvalue Equations} \label{AsymptoticEigenvalueEq}

The eigenvalue equation $Lw=\lambda w$ associated with \eqref{readifflinearoperator} is
\begin{equation}
w_{xx}+Df(\bar{u})w=\lambda w \label{assoeigen}.
\end{equation}
Written as a first-order system in the variable $W=(w,w^{\prime})^t$, this becomes
\begin{equation} \label{firstorderode}
W^{\prime}=\mathbb{A}(x;\lambda)W,
\end{equation}
where
$
\mathbb{A}(x;\lambda):=
  \begin{pmatrix}
   0 & I\\
   -(\lambda I+Df(\bar{u})) & 0
  \end{pmatrix}.
$
We begin by studying the limiting, constant coefficient systems $L_{\pm}w=\lambda w$ of \eqref{assoeigen} at $x=\pm \infty$,
\begin{equation}
w_{xx}+Df(u_{\pm})w=\lambda w, \label{asymode}
\end{equation}
or, written as a first-order system,
\begin{equation} \label{asymfirstorderode}
W^{\prime}=\mathbb{A}_{\pm}(\lambda)W,
\end{equation}
where
$
\mathbb{A}_{\pm}(\lambda):=
  \begin{pmatrix}
   0 & I\\
   -(\lambda I+Df(u_{\pm})) & 0
  \end{pmatrix}.
$

The normal modes of \eqref{asymfirstorderode} are $V_j^{\pm}e^{\mu_j^{\pm}x}$, $j=1,2,\ldots,2n$, where $\mu_j^{\pm}$, $V_j^{\pm}$ are the eigenvalues and eigenvectors of $\mathbb{A}_{\pm}$; these are easily seen to satisfy
$
V_j^{\pm}=
 \begin{pmatrix}
 v_j^{\pm}\\
 \mu_j^{\pm}v_j^{\pm}
 \end{pmatrix}, v_j^{\pm} \in \mathbb{C}^{n}
$
and
$[(\mu_j^{\pm})^2 I+\lambda I+Df(u_{\pm})]v_j^{\pm}=0$.
Let $\sigma_1^{\pm},\sigma_2^{\pm},\ldots,\sigma_n^{\pm}$ be the eigenvalues of $Df(u_{\pm})$. Then we have the equation
$
(\mu_j^{\pm})^2+\lambda=-\sigma_j^{\pm}
$,
and solve for $\mu_j^{\pm}$ to get
$\mu_j^{\pm}=\pm\sqrt{-\sigma_j^{\pm}-\lambda}$. Since we have assumed that $\mathrm{Re}\sigma_j^{\pm}<0$ for all $j=1,2,\ldots,n$, and $\lambda$ is on a contour which lies completely in the half plane $\{\mathrm{Re}\lambda <0\}$,
then one of $\sqrt{-\sigma_j^{\pm}-\lambda}$ and $-\sqrt{-\sigma_j^{\pm}-\lambda}$ has positive real part while the other has negative real part. After some rearrangement we have the following:

\begin{pro} \label{Proposition4.1}
For some $\eta^{\prime}>0$, there locally exist analytic choices\\
$\mathrm{Re}\mu_1^{\pm},\ldots,\mathrm{Re}\mu_n^{\pm}\leq -\eta^{\prime}<0
<\eta^{\prime}\leq\mathrm{Re}\mu_{n+1}^{\pm},\ldots,\mathrm{Re}\mu_{2n}^{\pm}$, and $V_1^{\pm},\ldots,V_{2n}^{\pm}$ for the eigenvalues and eigenvectors of $\mathbb{A}_{\pm}(\lambda)$, satisfying
$-\mathrm{Re}\mu_{j}^{\pm}=\mathrm{Re}\mu_{n+j}^{\pm}\geq \eta^{\prime}>0$ for $j=1,2,\ldots,n$ and
\begin{eqnarray*}
\mu_j^{\pm}(\lambda)=-\gamma_j^{\pm}-a_j^{\pm}\lambda-b_j^{\pm}\lambda^2+\mathcal{O}(\lambda^3),& &
\mu_{n+j}^{\pm}(\lambda)=\gamma_j^{\pm}+a_j^{\pm}\lambda+b_j^{\pm}\lambda^2+\mathcal{O}(\lambda^3),\\
V_j^{\pm}(\lambda)=\begin{pmatrix}r_j^{\pm}+\mathcal{O}(\lambda) \\ -\gamma_j^{\pm} r_j^{\pm}+\mathcal{O}(\lambda)\end{pmatrix},& &
V_{n+j}^{\pm}(\lambda)=\begin{pmatrix}r_j^{\pm}+\mathcal{O}(\lambda) \\ \gamma_j^{\pm} r_j^{\pm}+\mathcal{O}(\lambda)\end{pmatrix}.
\end{eqnarray*}
as $\lambda \to 0$, where $\gamma_j^{\pm}, a_j^{\pm}$ and $b_j^{\pm}$ are some constants such that $\mathrm{Re}\gamma_j^{\pm}\geq \eta^{\prime}>0$ for $j=1,2,\ldots,n$. $r_j^{\pm}$ are the right eigenvectors of $Df(u_{\pm})$ corresponding to eigenvalues $\sigma_j^{\pm}$.
\end{pro}

\begin{lem}
The adjoint of the operator $Lw=w_{xx}+Df(\bar{u})w$ is $L^{\ast}z=z_{xx}+zDf(\bar{u})$.
\end{lem}
\begin{proof}
Let $A(x)=Df(\bar{u}(x))$, then the conclusion follows from
\begin{eqnarray*}
\langle Lw,z\rangle&=&\int zLwdx=\int z(w_{xx}+Aw)dx=\int zw_{xx}dx+\int zAw dx\\
&=&\int z_{xx}wdx+\int zAwdx=\int(z_{xx}+zA)wdx=\langle w,L^{\ast}z\rangle.
\end{eqnarray*}
\end{proof}

({\bf Note:} Here, $w$ is an $n$-dimensional column vector and $z$ is an $n$-dimensional row vector.)

\begin{lem}
Let $H_\lambda(x,y)$ denote the Green function for the adjoint operator $(L-\lambda I)^\ast$ of $L-\lambda I$. Then,
$H_\lambda(x,y)=G_\lambda(y,x)^\ast$, where $M^{\ast}$ denotes $\overline{M^t}$ for a complex matrix $M$. In particular, for $y \neq x$ and $x$ fixed, the matrix $z(y)=G_\lambda(x,y)$ satisfies
 \begin{eqnarray}
 z_{yy}+zDf(\bar{u}(y))=\lambda z. \label{adjointeigen}
 \end{eqnarray}
\end{lem}
Consider \eqref{adjointeigen} as an ODE for general row vector $z$, or, written as a first order system,
\begin{eqnarray} \label{adjointfirstorderode}
Z^\prime=Z\tilde{\mathbb{A}}(x;\lambda),
\end{eqnarray}
where $Z=(z,z^{\prime})$ and
$
\tilde{\mathbb{A}}(x;\lambda):=
  \begin{pmatrix}
   0 & -(\lambda I+Df(\bar{u}))\\
   I & 0
  \end{pmatrix}.
$
\begin{lem} \label{Lemma_Adjoint}
$Z$ is a solution of \eqref{adjointfirstorderode} if and only if $Z\mathcal{S}W\equiv \rm{Constant}$ for any solution $W$ of \eqref{firstorderode}, where
$
\mathcal{S}=
  \begin{pmatrix}
   0 & I\\
   -I & 0
  \end{pmatrix}.
$
\end{lem}
\begin{proof}
 \begin{eqnarray*}
 (Z\mathcal{S}W)^{\prime}
 &=& (-z^{\prime}w+zw^{\prime})^{\prime}
 =-z^{\prime\prime}w-z^{\prime}w^{\prime}+z^{\prime}w^{\prime}+zw^{\prime\prime}\\
 &=& -z^{\prime\prime}w+zw^{\prime\prime}
 =-(\lambda z-zA)w+z(\lambda w-Aw)
 \\ &=&
 -\lambda zw+zAw+\lambda zw-zAw=0.
 \end{eqnarray*}
\end{proof}

Similarly, we define the adjoint asymptotic matrices
$$
\tilde{\mathbb{A}}_{\pm}(\lambda):=
  \begin{pmatrix}
   0 & -(\lambda I+Df(u_{\pm}))\\
   I & 0
  \end{pmatrix}.
$$

\begin{pro}
Under assumption $(\mathcal{D})$, $|\mathbb{A}(x;\lambda)-\mathbb{A}_{\pm}(\lambda)|\leq Ce^{-\eta|x|}$ as $x\to \pm\infty$ and
$|\tilde{\mathbb{A}}(x;\lambda)-\tilde{\mathbb{A}}_{\pm}(\lambda)|\leq Ce^{-\eta|x|}$ as $x\to \pm\infty$.
\end{pro}
\begin{proof}
	This follows immediately from \eqref{profileproperty}.
\end{proof}

\section{Construction of the Resolvent Kernel} \label{ConstructionoftheResolventKernel}

Define $$W_1^{+}(x;\lambda),W_2^{+}(x;\lambda),\ldots,W_{2n}^{+}(x;\lambda)$$ and $$W_1^{-}(x;\lambda),W_2^{-}(x;\lambda),\ldots,W_{2n}^{-}(x;\lambda)$$ as two bases of solutions to \eqref{firstorderode}, and $$\tilde{W}_1^{+}(x;\lambda),\tilde{W}_2^{+}(x;\lambda),\ldots,\tilde{W}_{2n}^{+}(x;\lambda)$$ and  $$\tilde{W}_1^{-}(x;\lambda),\tilde{W}_2^{-}(x;\lambda),\ldots,\tilde{W}_{2n}^{-}(x;\lambda)$$ as two bases of solutions to the adjoint first order ODE \eqref{adjointfirstorderode}, satisfying the relations
\begin{align} \label{duality}
\begin{aligned}
\tilde{W}_j^{+}\mathcal{S}W_k^{+}&=\delta_k^j\\
\tilde{W}_j^{-}\mathcal{S}W_k^{-}&=\delta_k^j
\end{aligned}
\end{align}

Applying the Gap Lemma of \cite{GZ,ZH} relating variable- to
constant-coefficient solutions, we obtain:
\begin{pro}
We have the following asymptotic description of $\tilde{W}_j^{\pm}$ and $W_k^{\pm}$:
\begin{eqnarray*}
\tilde{W}_j^{\pm}(x;\lambda)&=&\tilde{V}_j^{\pm}(\lambda)e^{-\mu_j^{\pm}(\lambda)x}(1+\mathcal{O}(e^{-(\eta/2)|x|})),
\quad \mathrm{for} \quad j=1,2,\ldots,2n;\\
W_k^{\pm}(x;\lambda)&=&V_k^{\pm}(\lambda)e^{\mu_k^{\pm}(\lambda)x}(1+\mathcal{O}(e^{-(\eta/2)|x|})),
\quad \mathrm{for} \quad k=1,2,\ldots,2n,
\end{eqnarray*}
where $V_j^\pm$, $\mu_j^\pm$ are as in Proposition \ref{Proposition4.1}.
\end{pro}


\begin{de}
We define the decaying modes of \eqref{firstorderode} as
\begin{eqnarray*}
\Phi^{+}&=&(\phi_1^{+},\ldots,\phi_n^{+})=(W_1^{+},\ldots,W_n^{+}),\\
\Phi^{-}&=&(\phi_1^{-},\ldots,\phi_n^{-})=(W_{n+1}^{-},\ldots,W_{2n}^{-}).
\end{eqnarray*}
and the growing modes of \eqref{firstorderode} as
\begin{eqnarray*}
\Psi^{+}&=&(\psi_1^{+},\ldots,\psi_n^{+})=(W_{n+1}^{+},\ldots,W_{2n}^{+}),\\
\Psi^{-}&=&(\psi_1^{-},\ldots,\psi_n^{-})=(W_{1}^{-},\ldots,W_{n}^{-}).
\end{eqnarray*}
Similarly we define the growing modes of \eqref{adjointfirstorderode} as
\begin{eqnarray*}
\tilde{\Phi}^{+}&=&(\tilde{\phi}_1^{+},\ldots,\tilde{\phi}_n^{+})^{t}=(\tilde{W}_1^{+},\ldots,\tilde{W}_n^{+})^{t},\\
\tilde{\Phi}^{-}&=&(\tilde{\phi}_1^{-},\ldots,\tilde{\phi}_n^{-})^{t}=(\tilde{W}_{n+1}^{-},\ldots,\tilde{W}_{2n}^{-})^{t}.
\end{eqnarray*}
and the decaying modes of \eqref{adjointfirstorderode} as
\begin{eqnarray*}
\tilde{\Psi}^{+}&=&(\tilde{\psi}_1^{+},\ldots,\tilde{\psi}_n^{+})^{t}=(\tilde{W}_{n+1}^{+},\ldots,\tilde{W}_{2n}^{+})^{t},\\
\tilde{\Psi}^{-}&=&(\tilde{\psi}_1^{-},\ldots,\tilde{\psi}_n^{-})^{t}=(\tilde{W}_{1}^{-},\ldots,\tilde{W}_{n}^{-})^{t}.
\end{eqnarray*}
\end{de}
From this definition we have,
\begin{eqnarray}
 \begin{array}{l l}
 \tilde{\phi}_{j}^{\pm}\mathcal{S}\phi_{k}^{\pm}=\delta_{k}^{j}; & \tilde{\phi}_{j}^{\pm}\mathcal{S}\psi_{k}^{\pm}=0,\\
 \tilde{\psi}_{j}^{\pm}\mathcal{S}\phi_{k}^{\pm}=0; & \tilde{\psi}_{j}^{\pm}\mathcal{S}\psi_{k}^{\pm}=\delta_{k}^{j}.
 \end{array}
\end{eqnarray}
or written as matrix form
$
 \begin{pmatrix}
 \tilde{\Phi}^{\pm}\\
 \tilde{\Psi}^{\pm}
 \end{pmatrix}\mathcal{S}
 \left( \Phi^{\pm}, \Psi^{\pm} \right)
=I.
$

\begin{lem} \label{Lemma_Jump}
 \begin{eqnarray*}
  \begin{bmatrix}
   G_\lambda & G_{\lambda,y}\\
   G_{\lambda,x} & G_{\lambda,xy}
  \end{bmatrix}_{(y)}
 =\begin{pmatrix}
   0 & -I\\
   I & 0
  \end{pmatrix}
 =\mathcal{S}^{-1},
 \end{eqnarray*}
where $[h(x)]_{(y)}$ denotes the jump in $h(x)$ at $x=y$, and $\mathcal{S}$ is as in $\mathrm{Lemma}$ $\ref{Lemma_Adjoint}$.
\end{lem}
\begin{proof}
Expanding $\delta_y(x)=(L-\lambda I)G_{\lambda}=G_{\lambda,xx}+AG_{\lambda}-\lambda G_{\lambda}$, and comparing orders of singularity, we find that $G_{\lambda,xx}(x,y)=\delta_y(x)$ and $AG_{\lambda}-\lambda G_{\lambda}=0$, thus
\begin{eqnarray*}
G_{\lambda,x}=H_{y}(x)
=\left\{
\begin{array}{l l}
I, \quad \text{for $x\geq y$;} \\  0, \quad \text{for $x<y$.}
\end{array}
\right.
\end{eqnarray*}
This gives $[G_{\lambda,x}]_{(y)}=I$ and $[G_{\lambda}]_{(y)}=0$. Differentiating $[G_{\lambda}]_{(y)}$ in $y$, we obtain
\begin{eqnarray*}
[G_{\lambda,x}]_{(y)}+[G_{\lambda,y}]_{(y)}=\frac{\dd}{\dd y}[G_{\lambda}]_{(y)}=0,
\end{eqnarray*}
thus $[G_{\lambda,y}]_{(y)}=-I$. Differentiating again we find that
$$
[G_{\lambda,xy}]_{(y)}=-\frac{1}{2}\left([G_{\lambda,xx}]_{(y)}+[G_{\lambda,yy}]_{(y)}\right).
$$
Finally we can determine $[G_{\lambda,xx}]_{(y)}$ and $[G_{\lambda,yy}]_{(y)}$ by
$G_{\lambda,xx}=\lambda G_{\lambda}-AG_{\lambda}$,
$G_{\lambda,yy}=\lambda G_{\lambda}-G_{\lambda}A$.
It is easy to find that $[G_{\lambda,xx}]_{(y)}=0$ and $[G_{\lambda,yy}]_{(y)}=0$, thus $[G_{\lambda,xy}]_{(y)}=0$.
\end{proof}

From $(L-\lambda I)G_\lambda(x,y)=\delta_y(x)I$, $(L-\lambda I)^\ast H_\lambda(x,y)=\delta_y(x)I$ we know that
$\begin{pmatrix}
G_\lambda(x,y)\\
G_{\lambda,x}(x,y)
\end{pmatrix}$
viewed as a function of $x$ satisfies \eqref{firstorderode}(differentiating with respect to $x$), while \\ $(G_\lambda(x,y),G_{\lambda,y}(x,y))$ viewed as function of $y$ satisfies \eqref{adjointfirstorderode}
(differentiating with respect to $y$). Furthermore, note that both $G_\lambda(x,\cdot)$ and $G_\lambda(\cdot,y)$ decay at $\pm\infty$ for $\lambda$ on the resolvent set, since $|(L-\lambda I)^{-1}|<\infty$ and $|(L-\lambda I)^{\ast -1}|<\infty$ imply $\|G_\lambda(\cdot,y)\|_{L^{1}(x)}<\infty$ and $\|G_\lambda(x,\cdot)\|_{L^{1}(y)}<\infty$ respectively. Combining, we have the representation
\begin{eqnarray} \label{greenrepresentation}
 \begin{pmatrix}
 G_\lambda & G_{\lambda,y}\\
 G_{\lambda,x} & G_{\lambda,xy}
 \end{pmatrix}
=\left\{
 \begin{array}{l l}
 \Phi^+(x;\lambda)M^+(\lambda)\tilde{\Psi}^-(y;\lambda) \quad \text{for $x>y$;}\\
 -\Phi^-(x;\lambda)M^-(\lambda)\tilde{\Psi}^+(y;\lambda) \quad \text{for $x<y$,}
 \end{array}
 \right.
\end{eqnarray}
where matrices $M^{\pm}(\lambda)$ are to be determined.

Combining Lemma $\ref{Lemma_Jump}$ and \eqref{greenrepresentation}, we have
\begin{eqnarray*}
\left(\Phi^+(y),\Phi^-(y)\right)
\begin{pmatrix}
M^+(\lambda) &             0\\
           0 & M^-(\lambda)
\end{pmatrix}
\begin{pmatrix}
\tilde{\Psi}^-(y)\\
\tilde{\Psi}^+(y)
\end{pmatrix}
=\mathcal{S}^{-1}, \quad \text{or}
\end{eqnarray*}
\begin{eqnarray*}
 \begin{pmatrix}
  M^+(\lambda) &             0\\
             0 & M^-(\lambda)
 \end{pmatrix}
&=& \left(\Phi^+(y),\Phi^-(y)\right)^{-1}\mathcal{S}^{-1}
  \begin{pmatrix}
  \tilde{\Psi}^-(y)\\
  \tilde{\Psi}^+(y)
  \end{pmatrix}^{-1}\\
&=& \left(
  \begin{pmatrix}
  \tilde{\Psi}^-(y)\\
  \tilde{\Psi}^+(y)
  \end{pmatrix}
  \mathcal{S}
  \left(\Phi^+(y),\Phi^-(y)\right)
    \right)^{-1}\\
&=& \begin{pmatrix}
    \tilde{\Psi}^-\mathcal{S}\Phi^+ & \tilde{\Psi}^-\mathcal{S}\Phi^-\\
    \tilde{\Psi}^+\mathcal{S}\Phi^+ & \tilde{\Psi}^+\mathcal{S}\Phi^-
    \end{pmatrix}^{-1}(y)\\
&=& \begin{pmatrix}
    \tilde{\Psi}^-\mathcal{S}\Phi^+ &                               0\\
                                  0 & \tilde{\Psi}^+\mathcal{S}\Phi^-
    \end{pmatrix}^{-1}(y)\\
&=& \begin{pmatrix}
    (\tilde{\Psi}^-\mathcal{S}\Phi^+)^{-1} &                                         0\\
                                            0 & (\tilde{\Psi}^+\mathcal{S}\Phi^-)^{-1}
    \end{pmatrix},
\end{eqnarray*}
thus $M^+(\lambda)=(\tilde{\Psi}^-\mathcal{S}\Phi^+)^{-1}$, $M^-(\lambda)=(\tilde{\Psi}^+\mathcal{S}\Phi^-)^{-1}$.

\begin{pro} \label{ProGreen}
On $\Omega$,
\begin{eqnarray}
 \begin{pmatrix}
 G_\lambda & G_{\lambda,y}\\
 G_{\lambda,x} & G_{\lambda,xy}
 \end{pmatrix}
=\sum_{j,k}M_{jk}^+(\lambda)\phi_{j}^+(x;\lambda)\tilde{\psi}_{k}^-(y;\lambda)
\end{eqnarray}
for $y \leq 0 \leq x$,
\begin{eqnarray}
 \begin{pmatrix}
 G_\lambda & G_{\lambda,y}\\
 G_{\lambda,x} & G_{\lambda,xy}
 \end{pmatrix}
=\sum_{j,k}d_{jk}^+(\lambda)\phi_{j}^-(x;\lambda)\tilde{\psi}_{k}^-(y;\lambda)
+\sum_{k}\psi_{k}^-(x;\lambda)\tilde{\psi}_{k}^-(y;\lambda)
\end{eqnarray}
for $y \leq x \leq 0$, where
\begin{equation}
M^+=(I,0)(\Phi^+,\Phi^-)^{-1}\Psi^-, \quad d^+=-(0,I)(\Phi^+,\Phi^-)^{-1}\Psi^{-}
\end{equation}
\begin{eqnarray}
 \begin{pmatrix}
 G_\lambda & G_{\lambda,y}\\
 G_{\lambda,x} & G_{\lambda,xy}
 \end{pmatrix}
=-\sum_{j,k}M_{jk}^-(\lambda)\phi_{j}^-(x;\lambda)\tilde{\psi}_{k}^+(y;\lambda)
\end{eqnarray}
for $x \leq 0 \leq y$,
\begin{eqnarray}
 \begin{pmatrix}
 G_\lambda & G_{\lambda,y}\\
 G_{\lambda,x} & G_{\lambda,xy}
 \end{pmatrix}
=\sum_{j,k}d_{jk}^-(\lambda)\phi_{j}^-(x;\lambda)\tilde{\psi}_{k}^-(y;\lambda)
-\sum_{k}\phi_{k}^-(x;\lambda)\tilde{\phi}_{k}^-(y;\lambda)
\end{eqnarray}
for $x \leq y \leq 0$, where
\begin{equation*}
M^-=\tilde{\Phi}^- \begin{pmatrix}\tilde{\Psi}^- \\ \tilde{\Psi}^+ \end{pmatrix}^{-1}
\begin{pmatrix}0 \\ I\end{pmatrix}
, \quad
d^-(\lambda)=\tilde{\Phi}^{-}\begin{pmatrix}\tilde{\Psi}^- \\ \tilde{\Psi}^+ \end{pmatrix}^{-1}
             \begin{pmatrix}I \\ 0 \end{pmatrix}
\end{equation*}
\end{pro}

\begin{proof}
Here we only prove the cases for $y \leq 0 \leq x$ and $y \leq x \leq 0$, the cases for $x \leq 0 \leq y$ and $x \leq y \leq 0$ can be derived similarly.

For $y\leq 0\leq x$, according to \eqref{greenrepresentation},
\begin{eqnarray} \label{greenkernelproof}
 \begin{pmatrix}
 G_\lambda & G_{\lambda,y}\\
 G_{\lambda,x} & G_{\lambda,xy}
 \end{pmatrix}
=\sum_{i=1}^{n}\sum_{j=1}^{n}M_{ij}^+(\lambda)
\phi_i^{+}(x;\lambda)\tilde{\psi}_j^{-}(y;\lambda)
\end{eqnarray}
We may express $M^+$ using the duality relation \eqref{duality} as
\begin{equation*}
\begin{aligned}
M^+&=(I,0)(\Phi^+,\Phi^-)^{-1}\mathcal{S}^{-1}
 \begin{pmatrix}
 \tilde{\Psi}^- \\
 \tilde{\Phi}^-
 \end{pmatrix}^{-1}
 \begin{pmatrix}
 I \\
 0
 \end{pmatrix}\\
 &=(I,0)(\Phi^+,\Phi^-)^{-1}\left(
 \begin{pmatrix}
 \tilde{\Psi}^- \\
 \tilde{\Phi}^-
 \end{pmatrix}
 \mathcal{S}
 \right)^{-1}
 \begin{pmatrix}
 I \\
 0
 \end{pmatrix}\\
&=(I,0)(\Phi^+,\Phi^-)^{-1}(\Psi^-,\Phi^-)
 \begin{pmatrix}
 I \\
 0
 \end{pmatrix}
 =(I,0)(\Phi^+,\Phi^-)^{-1}\Psi^-.
\end{aligned}
\end{equation*}

Next, for $y\leq x\leq 0$, expressing each $\phi_{i}^+(x;\lambda)$ as a linear combination of basis elements at $-\infty$,
\begin{equation} \label{phi_j+}
\phi_{i}^+(x;\lambda)=\sum_{j=1}^{n}a_{ji}^+(\lambda)\phi_{j}^-(x;\lambda)
+\sum_{j=1}^{n}b_{ji}^+(\lambda)\psi_{j}^-(x;\lambda),
\end{equation}
we plug this into \eqref{greenkernelproof} to derive
\begin{eqnarray*}
& &\sum_{i=1}^{n}\sum_{l=1}^{n}M_{il}^+(\lambda)\phi_{i}^{+}(x;\lambda)\tilde{\psi}_{l}^{-}(y;\lambda)\\
&=&\sum_{i=1}^{n}\sum_{l=1}^{n}M_{il}^+(\lambda)
   \sum_{j=1}^{n}a_{ji}^+(\lambda)\phi_{j}^{-}(x;\lambda)\tilde{\psi}_{l}^{-}(y;\lambda)\\
& &+\sum_{i=1}^{n}\sum_{l=1}^{n}M_{il}^+(\lambda)
   \sum_{j=1}^{n}b_{ji}^+(\lambda)\psi_{j}^{-}(x;\lambda)\tilde{\psi}_{l}^{-}(y;\lambda)\\
&=&\sum_{j=1}^{n}\sum_{l=1}^{n}\left(\sum_{i=1}^{n}a_{ji}^{+}(\lambda)M_{il}^{+}(\lambda)\right)
   \phi_{j}^{-}(x;\lambda)\tilde{\psi}_{l}^{-}(y;\lambda)\\
& &+\sum_{j=1}^{n}\sum_{l=1}^{n}\left(\sum_{i=1}^{n}b_{ji}^{+}(\lambda)M_{il}^{+}(\lambda)\right)
   \psi_{j}^{-}(x;\lambda)\tilde{\psi}_{l}^{-}(y;\lambda)\\
&=&\sum_{j=1}^{n}\sum_{l=1}^{k}d_{jl}^{+}(\lambda)\phi_{j}^{-}(x;\lambda)\tilde{\psi}_{l}^{-}(y;\lambda)
   +\sum_{j=1}^{n}\sum_{l=1}^{n}e_{jl}^{+}(\lambda)\psi_{j}^{-}(x;\lambda)\tilde{\psi}_{l}^{-}(y;\lambda).
\end{eqnarray*}
Here we are defining
\begin{eqnarray*}
d_{jl}^{+}(\lambda)&:=&\sum_{i=1}^{n}a_{ji}^{+}(\lambda)M_{il}^{+}(\lambda),\quad d^{+}=a^{+}M^{+};\\
e_{jl}^{+}(\lambda)&:=&\sum_{i=1}^{n}b_{ji}^{+}(\lambda)M_{il}^{+}(\lambda),\quad e^{+}=b^{+}M^{+}.
\end{eqnarray*}
where $a,b,d,e$ are all $n\times n$ matrices.

Rewriting \eqref{phi_j+} as
$\Phi^+=\Phi^- a^+ + \Psi^- b^+ =(\Phi^-, \Psi^-)\begin{pmatrix}a^+ \\ b^+\end{pmatrix}$,
and using the relation $\begin{pmatrix}\tilde{\Phi}^- \\ \tilde{\Psi}^-\end{pmatrix}\mathcal{S}(\Phi^-,\Psi^-)=I$,
we have
$
\begin{pmatrix}a^+ \\ b^+\end{pmatrix}=(\Phi^-,\Psi^-)^{-1}\Phi^+
=\begin{pmatrix}\tilde{\Phi}^- \\ \tilde{\Psi}^-\end{pmatrix}\mathcal{S}\Phi^+.
$

Finally, if we write $\Pi_{+}=(\Phi^+,0)(\Phi^+,\Phi^-)^{-1}$ and $\Pi_{-}=I-\Pi_{+}$ then
\begin{eqnarray*}
\begin{pmatrix}d^+ \\ e^+\end{pmatrix}&=&\begin{pmatrix}a^+ \\ b^+\end{pmatrix}M^+
=\begin{pmatrix}\tilde{\Phi}^- \\ \tilde{\Psi}^-\end{pmatrix}\mathcal{S}\Phi^+ (I,0)(\Phi^+,\Phi^-)^{-1}\Psi^-\\
&=&\begin{pmatrix}\tilde{\Phi}^- \\ \tilde{\Psi}^-\end{pmatrix}\mathcal{S}(\Phi^+,0)(\Phi^+,\Phi^-)^{-1}\Psi^-\\
&=&(\Phi^-,\Psi^-)^{-1}\Pi_+\Psi^-\\
&=&(\Phi^-,\Psi^-)^{-1}(I-\Pi_-)\Psi^-\\
&=&(\Phi^-,\Psi^-)^{-1}(I-(0,\Phi^-)(\Phi^+,\Phi^-)^{-1})\Psi^-\\
&=&(\Phi^-,\Psi^-)^{-1}\Psi^- -(\Phi^-,\Psi^-)^{-1}(0,\Phi^-)(\Phi^+,\Phi^-)^{-1}\Psi^-\\
&=&(\Phi^-,\Psi^-)^{-1}(\Phi^-,\Psi^-)\begin{pmatrix}0 \\ I\end{pmatrix}\\
& &-(\Phi^-,\Psi^-)^{-1}(\Phi^-,\Psi^-)\begin{pmatrix}0 & I \\ 0 & 0\end{pmatrix}(\Phi^+,\Phi^-)^{-1}\Psi^-\\
&=&\begin{pmatrix}0 \\ I\end{pmatrix}-\begin{pmatrix}0 & I \\ 0 & 0\end{pmatrix}(\Phi^+,\Phi^-)^{-1}\Psi^-
\end{eqnarray*}
Thus we have
$d^+=-(0,I)(\Phi^+,\Phi^-)^{-1}\Psi^-$, $e^+=I$,
and so we can rewrite \eqref{greenkernelproof} as
\begin{eqnarray*}
 \begin{pmatrix}
 G_\lambda & G_{\lambda,y}\\
 G_{\lambda,x} & G_{\lambda,xy}
 \end{pmatrix}
=\sum_{j,k}d_{jk}^+(\lambda)\phi_{j}^-(x;\lambda)\tilde{\psi}_{k}^-(y;\lambda)
+\sum_{k}\psi_{k}^-(x;\lambda)\tilde{\psi}_{k}^-(y;\lambda).
\end{eqnarray*}
\end{proof}

\section{Low and Bounded Frequency Bounds on the Resolvent Kernel} \label{LowandBoundedFrequencyBounds}

\begin{lem}
Under assumption $(\mathcal{D})$, for $|\lambda|\leq R$, any $R>0$, 
\begin{equation}
|M_{jk}^{\pm}(\lambda)|,|d_{jk}^{\pm}(\lambda)| \leq C,
\end{equation}
for a constant $C>0$ depending only on $R$.
\end{lem}
\begin{proof}
Expanding $M^+=(I,0)(\Phi^+,\Phi^-)^{-1}\Psi^-$ using Cramer's rule, and setting $x=0$, we obtain
$
M_{jk}^{+}=D^{-1}C_{jk}^{+},
$
where
$C^{+}=(I,0)\left(\Phi^+ ,\Phi^-\right)^{\mathrm{adj}}\Psi^{-}$ and $D=\det\left(\Phi^+ ,\Phi^-\right)$.
It is evident that $|C^{+}|$ is uniformly bounded and therefore
$|M_{jk}^{+}|\leq C_1|D|^{-1}\leq C$
by $(\mathcal{D})$, where $C$ is a uniform constant. Similar bounds hold for $M_{jk}^{-}$ and $d_{jk}^{\pm}$.
\end{proof}

\begin{pro}
Assuming $(\mathcal{D})$, for $|\lambda|\leq R$, any $R>0$, the resolvent kernel ${G}_{\lambda}$ satisfies the estimates
\begin{equation} \label{smalllambdaest}
|{G}_{\lambda}(x,y)|\leq Ce^{-\eta^{\prime}(|x|+|y|)}.
\end{equation}
where $\eta^{\prime}$ is as defined in Proposition $\ref{Proposition4.1}$ and $C>0$ is a constant which depends only on $R$.
\end{pro}
\begin{proof}
We only prove the case where $y\leq 0\leq x$, the rest is similar. According to Proposition $\ref{ProGreen}$, the Green kernel can be written as
\begin{eqnarray*}
& &\sum_{j,k}M_{jk}^+(\lambda)\phi_j^+(x;\lambda)\tilde{\psi}_k^-(y;\lambda)
=\sum_{j,k}M_{jk}^+(\lambda)W_j^+(x;\lambda)\tilde{W}_k^-(y;\lambda)\\
&=&\sum_{j,k}M_{jk}^+(\lambda)V_j^+(\lambda)e^{\mu_j^+(\lambda)x}\left(1+\mathcal{O}(e^{-(\eta/2)|x|})\right)
 \tilde{V}_k^-(\lambda)e^{-\mu_k^-(\lambda)y}\left(1+\mathcal{O}(e^{-(\eta/2)|y|})\right)\\
&=&\sum_{j,k}M_{jk}^+(\lambda)V_j^+(\lambda)\tilde{V}_k^-(\lambda)e^{-\gamma_j^+x}e^{\gamma_k^-y}
\left(1+\mathcal{O}(e^{-(\eta/2)|x|})\right)\left(1+\mathcal{O}(e^{-(\eta/2)|y|})\right)\\
&\leq&Ce^{-(\gamma_j^+-\eta^{\prime})x}e^{(\gamma_k^--\eta^{\prime})y}e^{-\eta^{\prime}x}e^{\eta^{\prime}y}
\leq Ce^{-\eta^{\prime}x}e^{\eta^{\prime}y}=Ce^{-\eta^{\prime}(|x|+|y|)}.
\end{eqnarray*}
\end{proof}

\section{High Frequency Bounds on the Resolvent Kernel} \label{HighFrequencyBounds}

Define $\Omega_{\theta}=\{\lambda:\mathrm{Re}(\lambda)\geq -\theta_1-\theta_2|\mathrm{Im}(\lambda)|\}$, for $\theta=(\theta_1,\theta_2)$ with $\theta_1,\theta_2>0$. Assuming $(\mathcal{D})$, we have the following estimates for $G_{\lambda}(x,y)$, given in \cite{ZH}.
\begin{pro}[Proposition 7.3, \cite{ZH}] \label{Pro7.1}
Under assumption $(\mathcal{D})$, it follows that for $R_0>0$ sufficiently large and $\theta_1,\theta_2>0$ sufficiently small there exist constants $C,\beta>0$ such that
\begin{equation}
\begin{aligned}
|G_{\lambda}(x,y)|&\leq C|\lambda|^{-1/2}e^{-\beta^{-1/2}|\lambda|^{1/2}|x-y|},\\
|G_{\lambda,x}(x,y)|,|G_{\lambda,y}(x,y)|&\leq Ce^{-\beta^{-1/2}|\lambda|^{1/2}|x-y|};
\end{aligned}
\end{equation}
for all $\lambda \in \Omega_{\theta}\setminus B(0,R_0)$.
\end{pro}

\section{Pointwise Bounds on the Green Function} \label{PointwiseBoundsontheGreenFunction}

Now we prove the pointwise bounds for $\tilde{G}(x,t;y)$ and $e(y,t)$ stated in Proposition $\ref{Proposition1.2}$.
\begin{proof}[Proof of Proposition \ref{Proposition1.2}]

To derive the bounds on $\tilde{G}(x,t;y)$, we consider two cases depending on the scale of $\frac{|x-y|}{t}$.
\begin{figure}[h]\centering
\begin{tikzpicture}[scale=1,extended line/.style={shorten >=-#1,shorten <=-#1},]
\draw [help lines] (-3,-3) grid (3,3);
\draw [->](0,-2.9)--(0,2.9) node[right]{$Im$};
\draw [->](-2.9,0)--(2.9,0) node[right]{$Re$};
\draw (-0.5,1)--(-0.5,-1);
\draw (-1,3)--(-0.5,1) node[right]{$-\frac{\eta}{2}+\kappa i$};
\draw (-1,-3)--(-0.5,-1) node[right]{$-\frac{\eta}{2}-\kappa i$};
\draw [<-](-0.75,2)--(-0.65,2.05) node[right]{$\tilde{\Gamma}_2$};
\draw [<-](-0.75,-2)--(-0.65,-2.05) node[right]{$\tilde{\Gamma}_2$};
\draw [<-](-0.5,0)--(-0.4,0.05) node[right]{$\tilde{\Gamma}_1$};
\end{tikzpicture}
\caption{Contour of integration}
\end{figure}
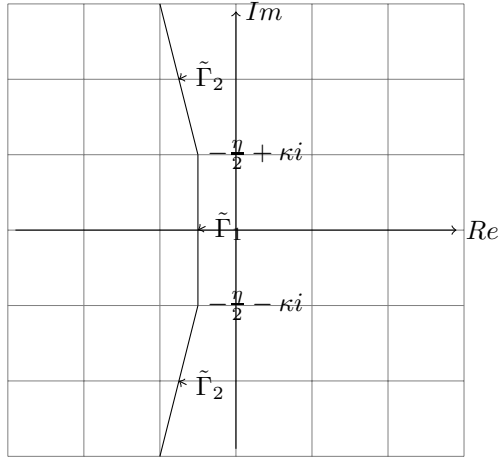

We define the contour $\tilde{\Gamma}$ as the union of bounded-$\lambda$ part $\tilde{\Gamma}_1$ and large-$\lambda$ part $\tilde{\Gamma}_2$, where $\tilde{\Gamma}_1$ is the line segment connecting $-\frac{\eta}{2}-\kappa i$ and $-\frac{\eta}{2}+\kappa i$, $\tilde{\Gamma}_2$ is the boundary of the sector $\Omega_{\theta}=\{\lambda:\mathrm{Re}(\lambda)\geq -\theta_1-\theta_2|\mathrm{Im}(\lambda)|\}$, $\theta_1=\frac{\eta}{4},\theta_2=\frac{\eta}{4\kappa}$ outside the ball
$B\left(0,\sqrt{(\frac{\eta}{2})^2+\kappa^2}\right)$.

{\bf Case I.} ($\frac{|x-y|}{t}$ large).
We first derive the bounds for the Green kernel $\tilde{G}(x,t;y)$ in the rather trivial case that 
\begin{align} \label{DefinitionofS}
\frac{|x-y|}{t} \geq S
\end{align}
for some $S>0$ sufficiently large, the regime in which standard short-time parabolic theory applies. Set
\begin{equation} \label{8.7}
\bar{\alpha}:=\frac{|x-y|}{2\beta t}, \quad R:= \beta \bar{\alpha}^2=\sqrt{\left(\frac{\eta}{2}\right)^2+\kappa^2},
\end{equation}
where $\beta$, $R_0$ are as in Proposition \ref{Pro7.1} and $S$ is sufficiently large that $R>R_0$, and consider again the representation of $\tilde{G}$:
\begin{equation*}
\begin{aligned}
\tilde{G}(x,t;y)
&=
\bar{u}^{\prime}(x)(1-\chi(t))\tilde{\psi}(y)
   +\frac{1}{2\pi i}\int_{\tilde{\Gamma}}e^{\lambda t}{G}_{\lambda}(x,y)d\lambda.\\
\end{aligned}
\end{equation*}
By the large $|\lambda|$ estimates of Proposition \ref{Pro7.1}, we have for all $\lambda \in \tilde{\Gamma}_2$ that
\begin{equation} \label{8.9}
|G_{\lambda}(x,y)|\leq C |\lambda|^{-1/2} e^{-\beta^{-\frac{1}{2}}|\lambda|^{\frac{1}{2}}|x-y|}.
\end{equation}
Further, we have
\begin{eqnarray} \label{8.10}
\mathrm{Re} \lambda &=&  -\frac{\eta}{2}, \quad \lambda\in \tilde{\Gamma}_1,\\
\mathrm{Re} \lambda &=& \mathrm{Re}\lambda_0 - \theta_2 (|\mathrm{Im} \lambda| - |\mathrm{Im} \lambda_0|), \quad \lambda \in \tilde{\Gamma}_2,
\end{eqnarray}
for $R$ sufficiently large, where $\lambda_0$ and $\lambda_0^*$ are the two points of intersection of $\tilde{\Gamma}_1$ and $\tilde{\Gamma}_2$.

Combining \eqref{smalllambdaest} and \eqref{8.10}, we obtain
\begin{eqnarray*}
\left|\int_{\tilde{\Gamma}_1}e^{\lambda t}G_{\lambda}(x,y)d\lambda\right|
&\leq&C\left|\int_{-\kappa}^{\kappa}e^{\left(-\frac{\eta}{2}+\xi i\right)t}e^{-\eta^{\prime}(|x|+|y|)}d\xi\right|
\leq Ce^{-\eta_0(t+|x-y|)}.
\end{eqnarray*}
Likewise,
\begin{eqnarray*}
& &\left|\int_{\tilde{\Gamma}_{2}}e^{\lambda t}G_{\lambda}(x,y)d\lambda\right|\\
&\leq&\int_{\tilde{\Gamma}_2}C|\lambda|^{-\frac{1}{2}}
      e^{(\mathrm{Re}\lambda)t-\beta^{-\frac{1}{2}}|\lambda|^{\frac{1}{2}}|x-y|}|d\lambda|\\
&\leq&Ce^{\mathrm{Re}(\lambda_0)t-\beta^{-\frac{1}{2}}|\lambda_0|^{\frac{1}{2}}|x-y|}
      \int_{\tilde{\Gamma}_2}|\lambda|^{-\frac{1}{2}}e^{(\mathrm{Re}\lambda-\mathrm{Re}\lambda_0)t}|d\lambda|\\
&\leq&Ce^{-\beta\bar{\alpha}^2t}\int_{\tilde{\Gamma}_2^{+}}
      |\mathrm{Im}\lambda-\mathrm{Im}\lambda_0|^{-\frac{1}{2}}e^{-\theta_2(\mathrm{Im}\lambda-\mathrm{Im}\lambda_0)t}
      |d(\mathrm{Im}\lambda-\mathrm{Im}\lambda_0)|\\
&=&Ce^{-\beta\bar{\alpha}^2t}\int_{0}^{+\infty}s^{-\frac{1}{2}}e^{-\theta_2ts}ds
      =Ct^{-\frac{1}{2}}e^{-\beta\bar{\alpha}^2t},
\end{eqnarray*}
which by \eqref{8.7} may be bounded by $Ct^{-\frac{1}{2}} e^{-\eta_0 t} e^{-\frac{(x-y)^{2}}{8\beta t}}$
for $\eta_0>0$ independent of $\bar{\alpha}$. 

Finally, since $(1-\chi(t))=0$ for $t\geq 2$, $|\bar u'(x)|\leq Ce^{-\eta_0 |x|}$,
and $|\tilde \psi(y)|\leq Ce^{-\eta_0 |y|}$, we have evidently
$|\bar{u}^{\prime}(x)(1-\chi(t))\tilde{\psi}(y)| \leq Ce^{-\eta_0(t+|x-y|)}$.

Combining the above three estimates, we have
$$
|\tilde{G}(x,t;y)|\leq
\left(C_1t^{-\frac{1}{2}}e^{-\eta_0 t}e^{-\frac{|x-y|^2}{4C_0t}}+C_2e^{-\eta_0(t+|x-y|)}\right)
$$
for $C_0$ sufficiently large.

{\bf Case II.} ($\frac{|x-y|}{t}$ bounded). In order to derive the bounds on $\tilde{G}(x,t;y)$ in this regime, 
we again recall the representation formula
\begin{eqnarray*}
\tilde{G}(x,t;y)&=&\bar{u}^{\prime}(x)(1-\chi(t))\tilde{\psi}(y)
   +\frac{1}{2\pi i}\int_{\tilde{\Gamma}}e^{\lambda t}G_{\lambda}(x,y)d\lambda,\\
\frac{1}{2\pi i}\int_{\tilde{\Gamma}}e^{\lambda t}G_{\lambda}(x,y)d\lambda
&=&\frac{1}{2\pi i}\int_{\tilde{\Gamma}_1}e^{\lambda t}G_{\lambda}(x,y)d\lambda
+\frac{1}{2\pi i}\int_{\tilde{\Gamma}_2}e^{\lambda t}G_{\lambda}(x,y)d\lambda.
\end{eqnarray*}

First, we estimate the $\tilde{\Gamma}_2$ part of $\tilde{G}$,
\begin{align*}
&\left|\frac{1}{2\pi i}\int_{\tilde{\Gamma}_{2}}e^{\lambda t}G_{\lambda}(x,y)d\lambda\right|\\
&\leq C\int_{\tilde{\Gamma}_2}e^{(\mathrm{Re}\lambda)t}|G_{\lambda}||d\lambda|
\leq C\int_{\tilde{\Gamma}_2}e^{(\mathrm{Re}\lambda_0)t-\theta_2(|\mathrm{Im}\lambda-\mathrm{Im}\lambda_0|)t}
      |\lambda|^{-\frac{1}{2}}|d\lambda|\\
&\leq Ce^{-\frac{1}{2}\eta t}\int_{\tilde{\Gamma}_2}|\mathrm{Im}\lambda-\mathrm{Im}\lambda_0|^{-\frac{1}{2}}
      e^{-\theta_2(|\mathrm{Im}\lambda-\mathrm{Im}\lambda_0|)t}|d(\mathrm{Im}\lambda-\mathrm{Im}\lambda_0)|\\
& =    2C(\theta_2t)^{-\frac{1}{2}}e^{-\frac{1}{2}\eta t}\Gamma\left(\frac{1}{2}\right)
      =    C_1t^{-\frac{1}{2}}e^{-\frac{1}{2}\eta t}
      \leq C_2t^{-\frac{1}{2}}e^{-\eta_0 t}e^{-\frac{|x-y|^2}{4C_0t}}.
\end{align*}
for $C_0>0$ large enough.

Next we estimate the $\tilde{\Gamma}_1$ part of $\tilde{G}$,
\begin{eqnarray*}
\left|\frac{1}{2\pi i}\int_{\tilde{\Gamma}_1}e^{\lambda t}G_{\lambda}(x,y)d\lambda\right|
&\leq&C\left|\int_{-\kappa}^{\kappa}e^{\left(-\frac{\eta}{2}+\xi i\right)t}e^{-\eta^{\prime}(|x|+|y|)}d\xi\right|\\
&=   &Ce^{-\frac{\eta}{2}t}e^{-\eta^{\prime}(|x|+|y|)}\left|\int_{-\kappa}^{\kappa}e^{i\xi t}d\xi\right|\\
&\leq&Ce^{-(\eta/2)t-\eta^{\prime}(|x|+|y|)}  \\
&\leq&Ce^{-\eta_0(t+|x|+|y|)} . 
\end{eqnarray*}
for $0<\eta_0$ less than $\frac{\eta}{2}$ and $\eta^{\prime}$.

Finally, we have as in the previous case $|\bar{u}^{\prime}(x)(1-\chi(t))\tilde{\psi}(y)| \leq Ce^{-\eta_0(t+|x-y|)}$.
Thus we know that $\tilde{G}(x,t;y)$ is bounded by
$C_1 t^{-\frac{1}{2}} e^{-\eta_0 t-\frac{|x-y|^2}{4C_0 t}}+C_2 e^{-\eta_0(t+|x-y|)}$ in both $\frac{|x-y|}{t}$ large and bounded cases.


This completes the proof of bounds on $\tilde{G}(x,t;y)$.
The bounds on $\tilde{G}_y(x,t;y)$ can be derived similarly. We just need to notice that in the estimate of $G_{\lambda,y}(x,y)$ for large-$\lambda$ is different from the same estimate of $G_{\lambda}(x,y)$ by a factor of $|\lambda|^{\frac{1}{2}}$, thus the large-$\lambda$ ($\tilde{\Gamma}_2$) part of the bounds on $\tilde{G}_y(x,t;y)$ is different from the $\tilde{G}(x,t;y)$ one by a factor of $t^{-\frac{1}{2}}$, while the bounded-$\lambda$ ($\tilde{\Gamma}_1$) part stays the same.

Next we move on to estimate $e(y,t)$. Recall that we have
$e(y,t)=\chi(t)\tilde{\psi}(y)$,
along with the estimates
\begin{eqnarray*}
\chi_{t}(t)&\leq& Ce^{-\eta_0 t},\text{ for } t\geq 0,\\
\tilde{\psi}(y)&\leq& Ce^{-\eta|y|},\text{ for } y \gtrless 0,\\
\tilde{\psi}_{y}(y)&\leq& Ce^{-\eta|y|},\text{ for } y \gtrless 0.
\end{eqnarray*}
Combining, we get the stated bounds for $e(y,t)$.
\end{proof}

\section{Improved Pointwise Bounds on the Green Function} \label{ImprovedPointwiseBoundsontheGreenFunction}

Next we prove the bounds stated in Proposition \ref{Proposition1.3}.
\begin{proof}[Proof of Proposition \ref{Proposition1.3}]

We first derive the bounds for the total Green kernel $G(x,t;y)$ in the rather trivial case that $\frac{|x-y|}{t} \geq S$, $S$ sufficiently large as defined in \eqref{DefinitionofS}, the regime in which standard short-time parabolic theory applies. Set
\begin{equation} \label{9.1}
\bar{\alpha}:=\frac{|x-y|}{2\beta t}, \quad R:= \beta \bar{\alpha}^2,
\end{equation}
where $\beta$, $R_0$ are as in Proposition \ref{Pro7.1} and $S$ is sufficiently large that $R>R_0$, and consider 
the representation of $G$ (following from \eqref{spectralresolution} and Cauchy's theorem):
$$
G(x,t;y)=\frac{1}{2\pi i}\int_{\Gamma_1\cup \Gamma_2}e^{\lambda t} G_{\lambda}(x,y) d\lambda,
$$
where $\Gamma_1:= \partial B(0,R)\cap \bar{\Omega}_\theta$ and $\Gamma_2:= \partial \Omega_\theta \setminus B(0,R)$.
Note that the intersection of $\Gamma$ with the real axis is $\lambda_{\mathrm{min}}=R=\beta \bar{\alpha}^2$.
By the large $|\lambda|$ estimates of Proposition \ref{Pro7.1}, we have for all $\lambda \in \Gamma_1\cup \Gamma_2$ that
\begin{equation} \label{9.2}
|G_{\lambda}(x,y)|\leq C |\lambda|^{-1/2} e^{-\beta^{-\frac{1}{2}}|\lambda|^{\frac{1}{2}}|x-y|}.
\end{equation}
Further, we have
\begin{eqnarray} \label{9.3}
\mathrm{Re} \lambda &\leq&  R(1- \eta_2\omega^2), \quad \lambda\in \Gamma_1,\\
\mathrm{Re} \lambda &\leq& \mathrm{Re}\lambda_0 - \theta_2 (|\mathrm{Im} \lambda| - |\mathrm{Im} \lambda_0|), \quad \lambda \in \Gamma_2,
\end{eqnarray}
for $R$ sufficiently large, where $\omega$ is the argument of $\lambda$ and $\lambda_0$ and $\lambda_0^*$ are the two points of intersection of $\Gamma_1$ and $\Gamma_2$,
for some $\eta_2>0$ independent of $\bar{\alpha}$.
Combining \eqref{9.2},\eqref{9.3} and \eqref{9.1}, we obtain
\begin{align*}
\left|\int_{\Gamma_{1}}e^{\lambda t}G_{\lambda}(x,y)d\lambda\right|
&\leq \int_{\Gamma_{1}}C|\lambda|^{-\frac{1}{2}}
      e^{(\mathrm{Re}\lambda)t-\beta^{-\frac{1}{2}}|\lambda|^{\frac{1}{2}}|x-y|}|d\lambda| \\
&\leq C\int_{\Gamma_1}R^{-\frac{1}{2}}e^{R(1-\eta_2\omega^2)t-\beta^{-\frac{1}{2}}R^{\frac{1}{2}}2\beta\bar{\alpha}t}|d\lambda|\\
&=CR^{\frac{1}{2}}e^{-\beta\bar{\alpha}^2t}\int_{-M}^{+M}e^{-R\eta_2t\omega^2}d\omega\\
&\leq C t^{-\frac{1}{2}}e^{-\beta \bar{\alpha}^{2}t}.
\end{align*}
Likewise,
\begin{eqnarray*}
    \left|\int_{\Gamma_{2}}e^{\lambda t}G_{\lambda}(x,y)d\lambda\right|
&\leq&\int_{\Gamma_2}C|\lambda|^{-\frac{1}{2}}
      e^{(\mathrm{Re}\lambda)t-\beta^{-\frac{1}{2}}|\lambda|^{\frac{1}{2}}|x-y|}|d\lambda|\\
&\leq&Ce^{\mathrm{Re}(\lambda_0)t-\beta^{-\frac{1}{2}}|\lambda_0|^{\frac{1}{2}}|x-y|}
      \int_{\Gamma_2}|\lambda|^{-\frac{1}{2}}e^{(\mathrm{Re}\lambda-\mathrm{Re}\lambda_0)t}|d\lambda|\\
&\leq&Ce^{Rt-\beta^{-\frac{1}{2}}R^{\frac{1}{2}}2\beta\bar{\alpha}t}
      2\int_{\Gamma_2^{+}}|\lambda|^{-\frac{1}{2}}e^{-\theta_2(\mathrm{Im}\lambda-\mathrm{Im}\lambda_0)t}|d\lambda|\\
&\leq&Ce^{-\beta\bar{\alpha}^2t}\int_{0}^{+\infty}s^{-\frac{1}{2}}e^{-\theta_2ts}ds
=Ct^{-\frac{1}{2}}e^{-\beta\bar{\alpha}^2t}.
\end{eqnarray*}

Combining these last two estimates, and recalling \eqref{9.1}, we have
$$
|G(x,t;y)|
\leq Ct^{-\frac{1}{2}} e^{-\frac{\beta \bar{\alpha}^{2}t}{2}} e^{-\frac{(x-y)^{2}}{8\beta t}}
\leq Ct^{-\frac{1}{2}} e^{-\eta_2 t} e^{-\frac{(x-y)^{2}}{8\beta t}},
$$
for $\eta_2>0$ independent of $\bar{\alpha}$, hence 
$|G(x,t;y)| \leq C t^{-\frac{1}{2}}e^{-\eta_0 t}e^{-\frac{|x-y|^2}{4C_0t}}$
for $\frac{|x-y|}{t} \geq S$ and $C_0$ sufficiently large.

Secondly, to prove the bound stated for $\tilde{H}$, recall that
\begin{eqnarray*}
F(x,t;y)&=&E(x,t;y)\left(\mathrm{errfn}\left(\frac{x-y+t}{\sqrt{4t}}\right)
 -\mathrm{errfn}\left(\frac{x-y-t}{\sqrt{4t}}\right)\right),\\
E(x,t;y)-F(x,t;y)&=&E(x,t;y)\left(1-\mathrm{errfn}\left(\frac{x-y+t}{\sqrt{4t}}\right)
+\mathrm{errfn}\left(\frac{x-y-t}{\sqrt{4t}}\right)\right).
\end{eqnarray*}
and
\begin{eqnarray}
\tilde{H}(x,t;y)&=&G(x,t;y)-F(x,t;y),\label{tilde_H_large}\\
\tilde{H}(x,t;y)&=&\tilde{G}(x,t;y)+(E(x,t;y)-F(x,t;y)).\label{tilde_H_bdd}
\end{eqnarray}

Here is the plan of the proof: For the case $\frac{|x-y|}{t}\geq S$ for some large enough $S$ defined in \eqref{DefinitionofS}, we use \eqref{tilde_H_large} with the bound on the total Green function $G(x,t;y)$ above and the bound on $F(x,t;y)$ that we are about to show. For the case $\frac{|x-y|}{t}\leq S$, we use \eqref{tilde_H_bdd} with the bound on $\tilde{G}(x,t;y)$ derived in the proof of Proposition \ref{Proposition1.2} and the bound on $E(x,t;y)-F(x,t;y)$ we are going to derive.

{\bf Case I. $\frac{|x-y|}{t}\leq \frac{1}{2}$.}
In this situation, we have $|x-y|\leq \frac{1}{2}t$, so
\begin{eqnarray*}
\frac{1}{4}\sqrt{t}&\leq \frac{x-y+t}{\sqrt{4t}} \leq& \frac{3}{4}\sqrt{t}\\
-\frac{3}{4}\sqrt{t}&\leq \frac{x-y-t}{\sqrt{4t}} \leq& -\frac{1}{4}\sqrt{t}
\end{eqnarray*}
Recall that $\mathrm{errfn}(x):=\frac{1}{\sqrt{\pi}}\int_{-\infty}^{x}e^{-z^2}dz$. From this we get
\begin{eqnarray*}
& &1-\mathrm{errfn}\left(\frac{x-y+t}{\sqrt{4t}}\right)
   +\mathrm{errfn}\left(\frac{x-y-t}{\sqrt{4t}}\right)\\
&=&\frac{1}{\sqrt{\pi}}\int_{-\infty}^{+\infty}e^{-z^2}dz-\frac{1}{\sqrt{\pi}}\int_{-\infty}^{\frac{x-y+t}{\sqrt{4t}}}e^{-z^2}dz
   +\frac{1}{\sqrt{\pi}}\int_{-\infty}^{\frac{x-y-t}{\sqrt{4t}}}e^{-z^2}dz\\
&=&\frac{1}{\sqrt{\pi}}\left(\int_{-\infty}^{+\infty}e^{-z^2}dz-\int_{-\infty}^{\frac{x-y+t}{\sqrt{4t}}}e^{-z^2}dz\right)
   +\frac{1}{\sqrt{\pi}}\int_{-\infty}^{\frac{x-y-t}{\sqrt{4t}}}e^{-z^2}dz\\
&=&\frac{1}{\sqrt{\pi}}\int_{\frac{x-y+t}{\sqrt{4t}}}^{+\infty}e^{-z^2}dz
   +\frac{1}{\sqrt{\pi}}\int_{-\infty}^{\frac{x-y-t}{\sqrt{4t}}}e^{-z^2}dz,
\end{eqnarray*}
and then
\begin{eqnarray*}
&    &\left|1-\mathrm{errfn}\left(\frac{x-y+t}{\sqrt{4t}}\right)
      +\mathrm{errfn}\left(\frac{x-y-t}{\sqrt{4t}}\right)\right|\\
&\leq&\frac{1}{\sqrt{\pi}}\int_{\frac{1}{4}\sqrt{t}}^{+\infty}e^{-z^2}dz
      +\frac{1}{\sqrt{\pi}}\int_{-\infty}^{-\frac{1}{4}\sqrt{t}}e^{-z^2}dz
=   \mathrm{erfc}\left(\frac{1}{4}\sqrt{t}\right)
\leq
e^{-\frac{1}{16}t},
\end{eqnarray*}
where we have used the fact that for the complementary error function $\mathrm{erfc}(x):=\frac{2}{\sqrt{\pi}}\int_{x}^{+\infty}e^{-z^2}dz$, there is the estimate $\mathrm{erfc}(x)\leq e^{-x^2}$. Together with the fact that $E(x,t;y)=\bar{u}^{\prime}(x)e(y,t)=\chi(t)\bar{u}^{\prime}(x)\tilde{\psi}(y)$, and $|\bar{u}^{\prime}(x)|\leq Ce^{-\eta |x|},|\tilde{\psi}(y)|\leq Ce^{-\eta |y|}$ for some $\eta>0$, we can derive that
\begin{eqnarray*}
& &\left|E(x,t;y)-F(x,t;y)\right|\\
&=   &\left|E(x,t;y)\left(1-\mathrm{errfn}\left(\frac{x-y+t}{\sqrt{4t}}\right)
      +\mathrm{errfn}\left(\frac{x-y-t}{\sqrt{4t}}\right)\right)\right|\\
&\leq&Ce^{-\eta|x|}e^{-\eta|y|}e^{-\frac{1}{16}t}
 \leq Ce^{-\eta|x-y|-\frac{1}{16}t}\\
&=   &Ct^{\frac{1}{2}}e^{\eta_0 t+\frac{|x-y|^2}{Mt^2}t-\eta\frac{|x-y|}{t}t-\frac{1}{16}t}\cdot t^{-\frac{1}{2}}e^{-\eta_0 t-\frac{|x-y|^2}{Mt^2}t}\\
&=   &Ct^{\frac{1}{2}}e^{-\left(\frac{1}{16}-\eta_0-\frac{|x-y|^2}{Mt^2}+\eta\frac{|x-y|}{t}\right)t}\cdot t^{-\frac{1}{2}}e^{-\eta_0 t-\frac{|x-y|^2}{Mt^2}t}\\
&\leq&Ct^{-\frac{1}{2}}e^{-\eta_0 t-\frac{|x-y|^2}{Mt}},
\end{eqnarray*}
for $0<\eta_0<\frac{1}{16}$ and $M>0$ large enough because $\frac{|x-y|}{t}$ is bounded. Together with the estimate $|\tilde{G}(x,t;y)|\leq Ct^{-\frac{1}{2}}e^{-\eta_0 t-\frac{|x-y|^2}{Mt}}$ for $\tilde{G}(x,t;y)$, 
as follows from $e^{-\eta|x-y|}\leq e^{-\frac{|x-y|^2}{St/\eta_0}}$ for $\frac{|x-y|}{t}<S$,
we derive that
$|\tilde{H}(x,t;y)|\leq Ct^{-\frac{1}{2}}e^{-\eta_0 t-\frac{|x-y|^2}{Mt}}.$

{\bf Case II. $\frac{1}{2}\leq\frac{|x-y|}{t}\leq S$.}
\begin{eqnarray*}
&    &\left|E(x,t;y)-F(x,t;y)\right|\\
&=   &\left|E(x,t;y)\left(1-\mathrm{errfn}\left(\frac{x-y+t}{\sqrt{4t}}\right)
      +\mathrm{errfn}\left(\frac{x-y-t}{\sqrt{4t}}\right)\right)\right|\\
&\leq&Ce^{-\eta|x|}e^{-\eta|y|}
 \leq Ce^{-\eta|x-y|}\\
&=   &Ct^{\frac{1}{2}}e^{\eta_0 t+\frac{|x-y|^2}{Mt^2}t-\eta\frac{|x-y|}{t}t}\cdot t^{-\frac{1}{2}}e^{-\eta_0 t-\frac{|x-y|^2}{Mt^2}t}\\
&=   &Ct^{\frac{1}{2}}e^{-\left(\eta\frac{|x-y|}{t}-\eta_0-\frac{|x-y|^2}{Mt^2}\right)t}\cdot t^{-\frac{1}{2}}e^{-\eta_0 t-\frac{|x-y|^2}{Mt^2}t}\\
&\leq&Ct^{-\frac{1}{2}}e^{-\eta_0 t-\frac{|x-y|^2}{Mt}},
\end{eqnarray*}
for $0<\eta_0<\frac{1}{2}\eta$ and $M>0$ large enough because $\frac{|x-y|}{t}$ is bounded. Together with the estimate $|\tilde{G}(x,t;y)|\leq Ct^{-\frac{1}{2}}e^{-\eta_0 t-\frac{|x-y|^2}{Mt}}$ for $\tilde{G}(x,t;y)$, 
as follows from $e^{-\eta|x-y|}\leq e^{-\frac{|x-y|^2}{St/\eta_0}}$ for $\frac{|x-y|}{t}<S$,
we derive that $|\tilde{H}(x,t;y)|\leq Ct^{-\frac{1}{2}}e^{-\eta_0 t-\frac{|x-y|^2}{Mt}}.$

{\bf Case III. $\frac{|x-y|}{t}\geq S$.}
\begin{eqnarray*}
\left|F(x,t;y)\right|
&=   &\left|E(x,t;y)\left(\mathrm{errfn}\left(\frac{x-y+t}{\sqrt{4t}}\right)
      -\mathrm{errfn}\left(\frac{x-y-t}{\sqrt{4t}}\right)\right)\right|\\
&\leq&Ce^{-\eta|x|}e^{-\eta|y|}\frac{1}{\sqrt{\pi}}\int_{\frac{x-y-t}{\sqrt{4t}}}^{\frac{x-y+t}{\sqrt{4t}}}e^{-z^2}dz\\
&\leq&Ce^{-\eta|x-y|}\sqrt{t}\max\left(e^{-\frac{(x-y+t)^2}{4t}},e^{-\frac{(x-y-t)^2}{4t}}\right)\\
&\leq&Ce^{-\eta\frac{|x-y|}{t}t}t^{\frac{1}{2}}\max\left(e^{-\frac{(x-y+t)^2}{4t}},e^{-\frac{(x-y-t)^2}{4t}}\right)\\
&\leq&Ct^{-\frac{1}{2}}e^{-\eta_0 t-\frac{|x-y|^2}{Mt}}.
\end{eqnarray*}
$S>\eta_0/\eta$ and $M>0$ large enough. 
In the second inequality above, note that $\frac{x-y+t}{\sqrt{4t}}$ and $\frac{x-y-t}{\sqrt{4t}}$ have the same sign, so we can estimate $\int_{\frac{x-y-t}{\sqrt{4t}}}^{\frac{x-y+t}{\sqrt{4t}}}e^{-z^2}dz$ by the width $\sqrt{t}$ of domain of integration times the maximum of integrand. Together with the estimate for the total Green function,
$|G(x,t;y)|\leq Ct^{-\frac{1}{2}} e^{-\eta_2 t} e^{-\frac{(x-y)^{2}}{8\beta t}}$, we derive that
$|\tilde{H}(x,t;y)|\leq Ct^{-\frac{1}{2}}e^{-\eta_0 t-\frac{|x-y|^2}{Mt}}.$

The proof of \eqref{tilde_e(y,t)estimate} is omitted (direct calculation).
\end{proof}

\section{Estimates on the Green Kernel} \label{LinearGreenEst}

Now we are ready to carry out the $L^p \to L^p$ estimation on the Green function $G(x,t;y)$.
\begin{pro}
Under assumption $(\mathcal{D})$, the Green function $G$ decomposes as $G(x,t;y)=\bar{u}^{\prime}(x)e(y,t)+\tilde{G}(x,t;y)$, where for some $C>0$, and all $t>0$,
$1\leq p\leq\infty$,
\begin{equation} \label{8.1}
\left|\int_{-\infty}^{+\infty}\tilde{G}(x,t;y)h(y)dy\right|_{L^{p}(x)}\leq Ce^{-\eta_0 t}|h|_{L^p},
\end{equation}
\begin{equation} \label{8.2}
\left|\int_{-\infty}^{+\infty}\tilde{G}_{y}(x,t;y)h(y)dy\right|_{L^{p}(x)}\leq Ct^{-\frac{1}{2}}e^{-\eta_0 t}|h|_{L^p},
\end{equation}
and
\begin{equation} \label{8.3}
\left|\int_{-\infty}^{+\infty}e(y,t)h(y)dy\right|\leq C|h|_{L^{p}},
\left|\int_{-\infty}^{+\infty}e_{y}(y,t)h(y)dy\right|\leq C|h|_{L^{p}},
\end{equation}
\begin{equation} \label{8.4}
\left|\int_{-\infty}^{+\infty}e_{t}(y,t)h(y)dy\right|\leq Ce^{-\eta_0 t}|h|_{L^{p}},
\left|\int_{-\infty}^{+\infty}e_{ty}(y,t)h(y)dy\right|\leq Ce^{-\eta_0 t}|h|_{L^{p}},
\end{equation}
for any $h \in L^p(\mathbb{R})$.
\end{pro}
\begin{proof}
First, we carry out the proof of \eqref{8.1}. Using Proposition $\ref{Proposition1.2}$, we have the following estimates on $|\tilde{G}(x,t;y)|_{L^1(x)}$:
\begin{eqnarray*}
\int_{-\infty}^{+\infty}t^{-\frac{1}{2}}e^{-\eta_0t-\frac{|x-y|^2}{4C_0t}}dx
=t^{-\frac{1}{2}}e^{-\eta_0t}\int_{-\infty}^{+\infty}e^{-\frac{|x-y|^2}{4C_0t}}dx
=Ce^{-\eta_0t},
\end{eqnarray*}
and
\begin{eqnarray*}
& &\int_{-\infty}^{+\infty}e^{-\eta_0(t+|x-y|)}dx=e^{-\eta_0t}\int_{-\infty}^{+\infty}e^{-\eta_0|x-y|}dx\\
&=&e^{-\eta_0t}\left(2\int_{0}^{+\infty}e^{-\eta_0\xi}d\xi\right)=Ce^{-\eta_0t},
\end{eqnarray*}
thus
\begin{eqnarray*}
|\tilde{G}(x,t;y)|_{L^1(x)}
\leq C\left|t^{-\frac{1}{2}}e^{-\eta_0t-\frac{|x-y|^2}{4C_0t}}\right|_{L^1(x)}
     +C\left|e^{-\eta_0(t+|x-y|)}\right|_{L^1(x)}
\leq Ce^{-\eta_0 t}.
\end{eqnarray*}
This implies that
\begin{eqnarray*}
\left|\int_{-\infty}^{+\infty}\tilde{G}(x,t;y)h(y)dy\right|_{L^{p}(x)}
\leq\sup_{y}|\tilde{G}(x,t;y)|_{L^1(x)}|h|_{L^p}
\leq Ce^{-\eta_0 t}|h|_{L^p},
\end{eqnarray*}
which completes the proof of \eqref{8.1}.

Next, in order to prove \eqref{8.3} and \eqref{8.4}, we note that $e(y,t)=\chi(t)\tilde{\psi}(y)$ and
$|\tilde{\psi}(y)|_{L^1\cup L^{\infty}}$ is bounded, so that gives us for $q=\frac{p}{p-1}$,
\begin{equation*}
\begin{aligned}
\left|\int_{-\infty}^{+\infty}e(y,t)h(y)dy\right|
=&   \left|\int_{-\infty}^{+\infty}\chi(t)\tilde{\psi}(y)h(y)dy\right|
     \leq C|\tilde{\psi}|_{L^{q}}|h|_{L^{p}}\leq C|h|_{L^{p}},\\
\left|\int_{-\infty}^{+\infty}e_y(y,t)h(y)dy\right|
=&   \left|\int_{-\infty}^{+\infty}\chi(t)\tilde{\psi}_y(y)h(y)dy\right|
     \leq C|\tilde{\psi}_y|_{L^{q}}|h|_{L^{p}}\leq C|h|_{L^{p}},\\
\left|\int_{-\infty}^{+\infty}e_t(y,t)h(y)dy\right|
=&   \left|\int_{-\infty}^{+\infty}\chi_t(t)\tilde{\psi}(y)h(y)dy\right|
     \leq Ce^{-\eta_0 t}|\tilde{\psi}|_{L^{q}}|h|_{L^{p}}\\
     &\leq Ce^{-\eta_0 t}|h|_{L^{p}},\\
\left|\int_{-\infty}^{+\infty}e_{ty}(y,t)h(y)dy\right|
=&   \left|\int_{-\infty}^{+\infty}\chi_t(t)\tilde{\psi}_y(y)h(y)dy\right|
     \leq Ce^{-\eta_0 t}|\tilde{\psi}_y|_{L^{q}}|h|_{L^{p}}\\
     &\leq Ce^{-\eta_0 t}|h|_{L^{p}}.
\end{aligned}
\end{equation*}
\end{proof}

\begin{pro} \label{Pro10.2}
Under assumption $(\mathcal{D})$, the Green function $G$ decomposes as $G(x,t;y)=\bar{u}^{\prime}(x)\tilde{e}(x,t;y)+\tilde{H}(x,t;y)$, where for some $C>0$, all $t>0$,
$1\leq p\leq\infty$, and $1 \leq p_0 \leq p$,
\begin{equation} \label{10.1}
\left|\int_{-\infty}^{+\infty}\tilde{H}(x,t;y)h(y)dy\right|_{L^{p}(x)}\leq Ce^{-\eta_0 t}|h|_{L^1},
\end{equation}
\begin{equation} \label{10.2}
\left|\int_{-\infty}^{+\infty}\tilde{H}(x,t;y)h(y)dy\right|_{L^{p}(x)}\leq Ce^{-\eta_0 t}|h|_{L^p},
\end{equation}
\begin{equation} \label{10.3}
\left|\int_{-\infty}^{+\infty}\tilde{H}(x,t;y)h(y)dy\right|_{L^{p}(x)}\leq Ce^{-\eta_0 t}|h|_{L^{p_0}},
\end{equation}
\begin{equation} \label{10.4}
\left|\int_{-\infty}^{+\infty}\tilde{H}(x,t;y)h(y)dy\right|_{L^{p}(x)}\leq Ce^{-\eta_0 t}|h|_{L^2},
\end{equation}
and
\begin{equation} \label{10.5}
\begin{aligned}
\left|\int_{-\infty}^{+\infty}\partial_{t}\tilde{e}(x,t;y)h(y)dy\right|_{L^{p}(x)}&\leq C(1+t)^{-\frac{1}{2}\left(1-\frac{1}{p}\right)-\frac{1}{2}}|h|_{L^{1}},\\
\left|\int_{-\infty}^{+\infty}\partial_{t}\partial_{x}^m\tilde{e}(x,t;y)h(y)dy\right|_{L^{p}(x)}&\leq C(1+t)^{-\frac{1}{2}\left(1-\frac{1}{p}\right)-\frac{m+1}{2}}|h|_{L^{1}},
\end{aligned}
\end{equation}
\begin{equation} \label{10.6}
\begin{aligned}
\left|\int_{-\infty}^{+\infty}\partial_{t}\tilde{e}(x,t;y)h(y)dy\right|_{L^{p}(x)}&\leq C(1+t)^{-\frac{1}{2}}|h|_{L^{p}},\\
\left|\int_{-\infty}^{+\infty}\partial_{t}\partial_{x}^m\tilde{e}(x,t;y)h(y)dy\right|_{L^{p}(x)}&\leq C(1+t)^{-\frac{m+1}{2}}|h|_{L^{p}},
\end{aligned}
\end{equation}
\begin{equation} \label{10.7}
\begin{aligned}
\left|\int_{-\infty}^{+\infty}\partial_{t}\tilde{e}(x,t;y)h(y)dy\right|_{L^{p}(x)}&\leq C(1+t)^{-\frac{1}{2}\left(\frac{1}{p_0}-\frac{1}{p}\right)-\frac{1}{2}}|h|_{L^{p_0}},\\
\left|\int_{-\infty}^{+\infty}\partial_{t}\partial_{x}^m\tilde{e}(x,t;y)h(y)dy\right|_{L^{p}(x)}&\leq C(1+t)^{-\frac{1}{2}\left(\frac{1}{p_0}-\frac{1}{p}\right)-\frac{m+1}{2}}|h|_{L^{p_0}},
\end{aligned}
\end{equation}
\begin{equation} \label{10.8}
\begin{aligned}
\left|\int_{-\infty}^{+\infty}\partial_{t}\tilde{e}(x,t;y)h(y)dy\right|_{L^{p}(x)}&\leq C(1+t)^{-\frac{1}{2}\left(\frac{1}{2}-\frac{1}{p}\right)-\frac{1}{2}}|h|_{L^{2}},\\
\left|\int_{-\infty}^{+\infty}\partial_{t}\partial_{x}^m\tilde{e}(x,t;y)h(y)dy\right|_{L^{p}(x)}&\leq C(1+t)^{-\frac{1}{2}\left(\frac{1}{2}-\frac{1}{p}\right)-\frac{m+1}{2}}|h|_{L^{2}},
\end{aligned}
\end{equation}
for all $h$ in the respective spaces $L^{1},L^{p},L^{p_0}$ and $L^2$ on the right-hand-side of each inequality.
\end{pro}
\begin{proof}
First, we prove inequality \eqref{10.1}. Recall that $\tilde{H}$ has the bound $$|\tilde{H}(x,t;y)|\leq Ct^{-\frac{1}{2}}e^{-\eta_0t-\frac{|x-y|^2}{Mt}},$$ we have the following estimate on $|\tilde{H}(x,t;y)|_{L^p}$,
\begin{eqnarray*}
& &\left(\int_{-\infty}^{+\infty}\left(t^{-\frac{1}{2}}e^{-\eta_0t-\frac{|x-y|^2}{4C_0t}}\right)^pdx\right)^{\frac{1}{p}}\\
&=&t^{-\frac{1}{2}}e^{-\eta_0t}\left(\int_{-\infty}^{+\infty}e^{-\frac{p|x-y|^2}{4C_0t}}dx\right)^{\frac{1}{p}}
 = t^{-\frac{1}{2}}e^{-\eta_0t}\left(2\int_{0}^{+\infty}e^{-\frac{p\xi^2}{4C_0t}}d\xi\right)^{\frac{1}{p}}\\
&=&t^{-\frac{1}{2}}e^{-\eta_0t}2^{\frac{1}{p}}\left(\int_0^{+\infty}e^{-\zeta^2}d\zeta\right)^{\frac{1}{p}}
   \left(\sqrt{\frac{4C_0t}{p}}\right)^{\frac{1}{p}}\\
&=&Ct^{-\frac{1}{2}\left(1-\frac{1}{p}\right)}e^{-\eta_0t}
\leq Ce^{-\eta_0t},
\end{eqnarray*}
for $t \geq 1$.
This implies that
\begin{eqnarray*}
&    &\left|\int_{-\infty}^{+\infty}\tilde{H}(x,t;y)h(y)dy\right|_{L^{p}(x)}\\
&\leq&\sup_{y}|\tilde{H}(x,t;y)|_{L^p(x)}\left|\int_{-\infty}^{+\infty}|h(y)|dy\right|
 \leq Ce^{-\eta_0 t}|h|_{L^1},
\end{eqnarray*}
proving \eqref{10.1}. Inequality \eqref{10.2} can be proved similarly, and inequality \eqref{10.3} can be obtained through $L^p$-interpolation.

Next, let us move on to the estimate of $\tilde{e}(x,t;y)$. It can be shown that $\partial_{t}\partial_{x}^m\tilde{e}(x,t;y)$ has the form
\begin{eqnarray*}
&&\chi(t)\tilde{\psi}(y)\left(\left(C_{m+1}\frac{(x-y+t)^m(x-y-t)}{t^{m+1+\frac{1}{2}}}+\cdots\right)e^{-\frac{(x-y+t)^2}{4t}}
\right.\\
&&+\left.\left(C^{\prime}_{m+1}\frac{(x-y-t)^m(x-y-3t)}{t^{m+1+\frac{1}{2}}}+\cdots\right)e^{-\frac{(x-y-t)^2}{4t}}\right),
\end{eqnarray*}
for $t\geq 1$ and $m \in \mathbb{N}$, we compute the following integral,
\begin{eqnarray*}
& &\left(\int_{-\infty}^{+\infty}\left(t^{-\left(m+1+\frac{1}{2}\right)}|x-y+t|^{m+1}
e^{-\frac{(x-y+t)^2}{4t}}\right)^pdx\right)^{\frac{1}{p}}\\
&=&t^{-\left(m+1+\frac{1}{2}\right)}\left(\int_{-\infty}^{+\infty}
|x-y+t|^{mp+p}e^{-\frac{p(x-y+t)^2}{4t}}dx\right)^{\frac{1}{p}}\\
&=&t^{-\left(m+1+\frac{1}{2}\right)}\left(2\int_{0}^{+\infty}\xi^{mp+p}e^{-\frac{p\xi^2}{4t}}d\xi\right)^{\frac{1}{p}}\\
&=&t^{-\left(m+1+\frac{1}{2}\right)}2^{\frac{1}{p}}\left(\int_0^{+\infty}\zeta^{mp+p}e^{-\zeta^2}d\zeta\right)^{\frac{1}{p}}
   \left(\sqrt{\frac{4t}{p}}\right)^{m+1+\frac{1}{p}}\\
&=&Ct^{-\frac{1}{2}\left(1-\frac{1}{p}\right)-\frac{m+1}{2}}
\leq C(1+t)^{-\frac{1}{2}\left(1-\frac{1}{p}\right)-\frac{m+1}{2}},
\end{eqnarray*}
and similarly
\begin{eqnarray*}
& &\left(\int_{-\infty}^{+\infty}\left(t^{-\left(m+1+\frac{1}{2}\right)}|x-y-t|^{m+1}
e^{-\frac{|x-y-t|^2}{4t}}\right)^pdx\right)^{\frac{1}{p}}\\
&\leq& C(1+t)^{-\frac{1}{2}\left(1-\frac{1}{p}\right)-\frac{m+1}{2}},
\end{eqnarray*}
thus we have
\begin{eqnarray*}
&    &\left|\int_{-\infty}^{+\infty}\partial_{t}\partial_{x}^m\tilde{e}(x,t;y)h(y)dy\right|_{L^{p}(x)}\\
&\leq&\sup_{y}|\partial_{t}\partial_{x}^m\tilde{e}(x,t;y)|_{L^p(x)}\left|\int_{-\infty}^{+\infty}|h(y)|dy\right|
 \leq C(1+t)^{-\frac{1}{2}\left(1-\frac{1}{p}\right)-\frac{m+1}{2}}|h|_{L^1},
\end{eqnarray*}
proving \eqref{10.5}. Inequality \eqref{10.6} can be proved similarly, and inequality \eqref{10.7} can be obtained through $L^p$-interpolation.
\end{proof}

\section{Integral Representation for $L^p$ Iteration Scheme} \label{Section_Integral_rep}

Letting $\tilde{u}$ be a second solution of \eqref{readiff}, define the perturbation
\begin{equation}
u(x,t):=\tilde{u}(x+\alpha(t),t)-\bar{u}(x)
\end{equation}
as the difference between a translate by $\alpha(t)$ of $\tilde{u}$ and the background wave $\bar{u}$. This yields after a brief computation the perturbation equation
\begin{equation} \label{perturbationeq}
\begin{aligned}
u_t-Lu&=f(u+\bar{u})+\bar{u}_{xx}-Df(\bar{u})u+\dot{\alpha}(t)(u_x+\bar{u}_x)\\
&=f(u+\bar{u})-f(\bar{u})-Df(\bar{u})u+\dot{\alpha}(t)(u_x+\bar{u}_x)\\
&=:N(u,\bar{u})+\dot{\alpha}(t)(u_x+\bar{u}_x)
\end{aligned}
\end{equation}
where $Lu=u_{xx}+Df(\bar{u})u$ and $N(u,\bar{u})=f(u+\bar{u})-f(\bar{u})-Df(\bar{u})u$.

We next choose $\alpha$ implicitly so as to ensure decay of $u$, i.e., to cancel the non-decaying linear translational effects encoded in term $\bar u'(x) e(y,t)$
of the Green kernel.
Noting that $e(y,0)=0$, we set $\alpha(0)=0$.
Applying Duhamel's principle to \eqref{perturbationeq}, we thus obtain
\begin{eqnarray*}
u(x,t)
&=&\int_{-\infty}^{+\infty}G(x,t;y)u_0(y)dy\\
& &+\int_{0}^{t}\int_{-\infty}^{+\infty}G(x,t-s;y)
   [N(u,\bar{u})+\dot{\alpha}(u_x+\bar{u}_x)](y,s)dyds\\
&=&\int_{-\infty}^{+\infty}G(x,t;y)u_0(y)dy\\
& &+\int_{0}^{t}\int_{-\infty}^{+\infty}G(x,t-s;y)
   [N(u,\bar{u})(y,s)+\dot{\alpha}(s)u_x(y,s)]dyds\\
& &+\alpha(t)\bar{u}^{\prime}(x)
\end{eqnarray*}
where $u_0(x):=u(x,0)$.
Here, we have used 
$$
\int_{-\infty}^{+\infty}G(x,t-s;y)\bar{u}^{\prime}(y)dy=e^{Lt}\bar{u}^{\prime}(x)=\bar{u}^{\prime}(x)
$$
and the normalization $\alpha(0)=0$.
Expanding $G(x,t;y)$ using \eqref{greenfunctiondecom}, we obtain
\begin{equation*}
\begin{aligned}
&u(x,t)\\
=&\int_{-\infty}^{+\infty}(\bar{u}^{\prime}(x)e(y,t)+\tilde{G}(x,t;y))u_0(y)dy\\
& +\int_{0}^{t}\int_{-\infty}^{+\infty}(\bar{u}^{\prime}(x)e(y,t-s)+\tilde{G}(x,t-s;y))
   [N(u,\bar{u})(y,s)+\dot{\alpha}(s)u_x(y,s)]dyds\\
& +\alpha(t)\bar{u}^{\prime}(x)\\
=&\int_{-\infty}^{+\infty}\tilde{G}(x,t;y)u_0(y)dy\\
& +\int_{0}^{t}\int_{-\infty}^{+\infty}\tilde{G}(x,t-s;y)
   [N(u,\bar{u})(y,s)+\dot{\alpha}(s)u_x(y,s)]dyds\\
& +\bar{u}^{\prime}(x)\left(\alpha(t)+\int_{-\infty}^{+\infty}e(y,t)u_0(y)dy\right.\\
& +\left.\int_{0}^{t}\int_{-\infty}^{+\infty}e(y,t-s)[N(u,\bar{u})(y,s)+\dot{\alpha}(s)u_x(y,s)]dyds\right)
\end{aligned}
\end{equation*}
Thus, if we define $\alpha(t)$ as
\begin{equation} \label{translate}
\begin{aligned}
\alpha(t)
=&-\int_{-\infty}^{+\infty}e(y,t)u_0(y)dy\\
&-\int_{0}^{t}\int_{-\infty}^{+\infty}e(y,t-s)[N(u,\bar{u})(y,s)+\dot{\alpha}(s)u_x(y,s)]dyds\\
=&-\int_{-\infty}^{+\infty}e(y,t)u_0(y)dy-\int_{0}^{t}\int_{-\infty}^{+\infty}e(y,t-s)N(u,\bar{u})(y,s)dyds\\
&+\int_{0}^{t}\int_{-\infty}^{+\infty}e_{y}(y,t-s)\dot{\alpha}(s)u(y,s)dyds,
\end{aligned}
\end{equation}
we obtain the integral representations
\begin{equation} \label{integral_rep_perturb}
\begin{aligned}
u(x,t)
=&\int_{-\infty}^{+\infty}\tilde{G}(x,t;y)u_0(y)dy
  +\int_{0}^{t}\int_{-\infty}^{+\infty}\tilde{G}(x,t-s;y)N(u,\bar{u})(y,s)dyds\\
&-\int_{0}^{t}\int_{-\infty}^{+\infty}\tilde{G}_y(x,t-s;y)\dot{\alpha}(s)u(y,s)dyds
\end{aligned}
\end{equation}
and
\begin{equation} \label{translate_deri}
\begin{aligned}
\dot{\alpha}(t)
=&-\int_{-\infty}^{+\infty}e_{t}(y,t)u_0(y)dy-\int_{0}^{t}\int_{-\infty}^{+\infty}e_{t}(y,t-s)N(u,\bar{u})(y,s)dyds\\
&+\int_{0}^{t}\int_{-\infty}^{+\infty}e_{ty}(y,t-s)\dot{\alpha}(s)u(y,s)dyds.
\end{aligned}
\end{equation}
Note that \eqref{translate} yields $\alpha(0)=0$, consistent with the derivation, hence \eqref{integral_rep_perturb}--\eqref{translate_deri} are indeed equivalent to
the original PDE.


\section{$L^p$ Nonlinear Iteration and $L^p$ Nonlinear Stability} \label{LpNonlinearIterationandLpNonlinearStability}

Associated with the solution $(u,\dot{\alpha})$ of the integral system \eqref{translate} and \eqref{integral_rep_perturb}, we define
\begin{equation} \label{zeta}
\zeta(t):=\sup_{0\leq s\leq t,p_0\leq p\leq \infty}\left(|u(x,s)|_{L^p(x)}+|\dot{\alpha}(s)|\right)e^{\eta_0 s}.
\end{equation}
\begin{lem}
For all $t\geq 0$ for which $\zeta(t)$ is finite, we have the estimate
\begin{equation}
\zeta(t)\leq C(E_0+\zeta(t)^2) \label{zeta_est}
\end{equation}
for some constant $C>0$, so long as $E_0$ is sufficiently small, where $E_0$ is defined as
$$E_0:=|u_0|_{L^{p_0}\cap L^{\infty}}=|u(x,0)|_{L^{p_0}(x)\cap L^{\infty}(x)}=|\tilde{u}-\bar{u}|_{L^{p_0}\cap L^{\infty}}|_{t=0}.$$
\end{lem}
\begin{proof}
Use Taylor expansion of $f(u+\bar{u})$ in $N(u,\bar{u})=f(u+\bar{u})-f(\bar{u})-Df(\bar{u})u$ to derive
$N(u,\bar{u})=\mathcal{O}(|u|^2)$.

We have then the following estimates of $|N(u,\bar{u})|_{L^p}(s)$ and $|\dot{\alpha}(s)u(y,s)|_{L^p(y)}$,
\begin{eqnarray*}
|N(u,\bar{u})(y,s)|_{L^{p}(y)}\leq& C|u|_{L^p}(s)|u|_{L^{\infty}}(s)&\leq C\zeta^2(s)e^{-2\eta_0 s}\\
|\dot{\alpha}(s)u(y,s)|_{L^{p}(y)}\leq& C|u|_{L^{p}}(s)|\dot{\alpha}(s)|&\leq C\zeta^2(s)e^{-2\eta_0 s}
\end{eqnarray*}
Using these together with our bounds on $\tilde{G}$ and $e$, we can now estimate $|u(\cdot,t)|_{L^p(x)}$. Using the
representation \eqref{integral_rep_perturb} of $u(x,t)$ together with estimates \eqref{8.1} and \eqref{8.2},
\begin{align*}
|u(\cdot,t)|_{L^p(x)}
&\leq Ce^{-\eta_0 t} E_0+C\int_{0}^{t}e^{-\eta_0(t-s)}|N(u,\bar{u})|_{L^p}(s)ds\\
&\quad+C\int_{0}^{t}e^{-\eta_0(t-s)}|\dot{\alpha}(s)u(y,s)|_{L^p}(s)ds\\
&\leq Ce^{-\eta_0 t} E_0+C\zeta^2(t)\int_{0}^{t}e^{-\eta_0(t-s)}e^{-2\eta_0 s}ds\\
&\quad+C\zeta^2(t)\int_{0}^{t}(t-s)^{-\frac{1}{2}}e^{-\eta_0(t-s)}e^{-2\eta_0 s}ds\\
&\leq C(E_0+\zeta(t)^2)e^{-\eta_0 t}.
\end{align*}
Similarly, for $|\dot{\alpha}(t)|$, using \eqref{translate_deri} together with \eqref{8.4} we have,
\begin{align*}
|\dot{\alpha}(t)|
&\leq Ce^{-\eta_0 t} E_0+C\int_{0}^{t}e^{-\eta_0(t-s)}|N(u,\bar{u})|_{L^{p}}(s)ds\\
&\quad+C\int_{0}^{t}e^{-\eta_0(t-s)}|\dot{\alpha}(s)u(y,s)|_{L^{p}}(s)ds\\
&\leq Ce^{-\eta_0 t} E_0+C\zeta^2(t)\int_{0}^{t}e^{-\eta_0(t-s)}e^{-2\eta_0 s}ds\\
&\quad+C\zeta^2(t)\int_{0}^{t}e^{-\eta_0(t-s)}e^{-2\eta_0 s}ds\\
&\leq C(E_0+\zeta(t)^2)e^{-\eta_0 t}.
\end{align*}
Rearranging the above two estimates together we obtain \eqref{zeta_est}.
\end{proof}

Finally, we give a proof of Theorem \ref{mainthm}.
\begin{proof}[Proof of Theorem \ref{mainthm}]
The first two bounds are proved by continuous induction. Taking $E_0< \frac{1}{4C^2}$, we have that $\zeta(t)< 2CE_0$ whenever $\zeta(t) \le 2CE_0$, and so the set of $t\ge 0$ for which $\zeta(t)< 2CE_0$ is equal to the set of $t\ge 0$ for which $\zeta(t)\le 2CE_0$. Recalling that $\zeta$ is continuous wherever it is finite, we find that the set of $t\ge 0$ for which $\zeta(t)< 2CE_0$ is both open and closed. Taking without loss of generality $C>1/2$, so that $t=0$ is contained in this set, then the set is nonempty. It follows that $\zeta(t)< 2CE_0$ for all $t\ge 0$, yielding the first two bounds.

The third follows using \eqref{translate} together with \eqref{8.3},
\begin{align*}
|\alpha(t)|
&\leq C E_0+C\int_{0}^{t}|N(u,\bar{u})|_{L^{p}}(s)ds+C\int_{0}^{t}|\dot{\alpha}(s)u(y,s)|_{L^{p}}(s)ds\\
&\leq C E_0+C\zeta^2(t)\int_{0}^{t}e^{-2\eta_0 s}ds+C\zeta^2(t)\int_{0}^{t}e^{-2\eta_0 s}ds\\
&\leq C E_0+C_1 E_0^2\leq C_2 E_0.
\end{align*}

To show the last inequality, notice that
$$\tilde{u}(x,t)-\bar{u}(x)=u(x-\alpha(t),t)-(\bar{u}(x)-\bar{u}(x-\alpha(t))),$$
so that $|\tilde{u}(\cdot,t)-\bar{u}|$ is controlled by the sum of $|u|$ and
$\bar{u}-\bar{u}(x-\alpha(t))=\mathcal{O}(\alpha(t)|\bar{u}^{\prime}(x)|)$, hence remains $\leq CE_0$ for all $t\geq 0$, for $E_0$
sufficiently small.
\end{proof}

\begin{cor}
The translate function $\alpha(t)$ in \eqref{translate} converges to a limit $\alpha_{\infty}$ as $t \to \infty$, and we have the following estimates,
\begin{eqnarray*}
|\alpha(t)-\alpha_{\infty}|&\leq&Ce^{-\eta_0 t}|\tilde{u}-\bar{u}|_{L^{p_0}\cap L^{\infty}}|_{t=0},\\
|\tilde{u}(x,t)-\bar{u}(x-\alpha_{\infty})|_{L^p(x)}&\leq& Ce^{-\eta_0 t}|\tilde{u}-\bar{u}|_{L^{p_0}\cap L^{\infty}}|_{t=0}.
\end{eqnarray*}
\end{cor}
\begin{proof}
Take a sequence $0<t_0<t_1<\ldots<t_n<t_{n+1}<\ldots$ such that $\lim_{n\to\infty}t_n=\infty$, then we have for $m<n,m,n\in \mathbb{N}$,
\begin{equation*}
\begin{aligned}
|\alpha(t_n)-\alpha(t_m)|&\leq |\dot{\alpha}((1-\theta)t_m+\theta t_n)||t_n-t_m|\\
&\leq Ce^{-\eta_0 ((1-\theta)t_m+\theta t_n)}|t_n-t_m|\to 0
\end{aligned}
\end{equation*}
as $m,n\to\infty$, thus $\{\alpha(t_n)\}$ is a Cauchy sequence hence there exist a $\alpha_{\infty}$ such that $$\lim_{n\to\infty}\alpha(t_n)=\alpha_{\infty}.$$
This shows the existence of $\alpha_{\infty}$, then we prove the first inequality. Indeed,
\begin{eqnarray*}
|\alpha(t)-\alpha_{\infty}|=\left|\int_{t}^{\infty}\dot{\alpha}(s)ds\right|\leq\left|\int_{t}^{\infty}CE_0 e^{-\eta_0 s}ds\right|\leq C_3E_0 e^{-\eta_0 t}.
\end{eqnarray*}
The second inequality follows using Theorem $\ref{mainthm}$,
\begin{align*}
&\quad|\tilde{u}(x,t)-\bar{u}(x-\alpha_{\infty})|_{L^p(x)}\\
&\leq|\tilde{u}(x,t)-\bar{u}(x-\alpha(t))+\bar{u}(x-\alpha(t))-\bar{u}(x-\alpha_{\infty})|_{L^p(x)}\\
&\leq|\tilde{u}(x,t)-\bar{u}(x-\alpha(t))|_{L^p(x)}+|\bar{u}(x-\alpha(t))-\bar{u}(x-\alpha_{\infty})|_{L^p(x)}\\
&=   |\tilde{u}(x,t)-\bar{u}(x-\alpha(t))|_{L^p(x)}+|\bar{u}^{\prime}(x_{\theta})(\alpha(t)-\alpha_{\infty})|_{L^p(x)}\\
&\leq CE_0 e^{-\eta_0 t}+ C_4E_0 e^{-\eta_0 t}\leq CE_0 e^{-\eta_0 t}.
\end{align*}
This completes the proof of the Corollary.
\end{proof}

\section{Integral Representation for $H^K$ and Pointwise Iteration Schemes} \label{IntegralRepresentationforHKandPointwiseIterationSchemes}

Let $\tilde{u}(x,t)$ be a solution of the system of reaction diffusion equations
$$
u_t=u_{xx}+f(u)
$$
and define $u(x,t)=\tilde{u}(x+\tilde{\alpha}(x,t),t)$ for some unknown function $\tilde{\alpha}:\mathbb{R}^2\to\mathbb{R}$
to be determined later. Moreover, let $\bar{u}(x)$ be a stationary solution and define
\begin{equation}
v(x,t)=u(x,t)-\bar{u}(x)=\tilde{u}(x+\tilde{\alpha}(x,t),t)-\bar{u}(x) \label{pertvar}
\end{equation}

\begin{lem}
For $v$, $u$ as above, we have
\begin{equation}\label{eqn:1nlper}
\begin{aligned}
u_t-u_{xx}-f(u)&=\left(\partial_t-L\right)\bar{u}^{\prime}(x)\tilde{\alpha}(x,t)
+\partial_x R\\
&\quad+(\partial_t+\partial_x^2)S
+\left(f(v(x,t)+\bar{u}(x))-f(\bar{u}(x))\right)\tilde{\alpha}_x,
\end{aligned}
\end{equation}
where
\begin{equation*}
\begin{aligned}
R:&= v\tilde{\alpha}_t + v\tilde{\alpha}_{xx}+  (\bar u_x(x) +v_x(x,t))\frac{\tilde{\alpha}_x^2}{1+\tilde{\alpha}_x}\\
&=\mathcal{O}\left(|v|(|\tilde{\alpha}_t|+|\tilde{\alpha}_{xx}|) +\left(\frac{|\bar u_x|+|v_x|}{1-|\tilde{\alpha}_x|} \right)|\tilde{\alpha}_x|^2\right)
\end{aligned}
\end{equation*}
and
$$
S:=- v\tilde{\alpha}_x =\mathcal{O}\left(|v|\cdot|\tilde{\alpha}_x|\right).
$$
\end{lem}

\begin{proof}
Using the fact that
$\tilde u_t-\tilde{u}_{xx}-f(\tilde{u})=0$, it follows by a straightforward computation that
\begin{equation}\label{altform}
u_t-f(u)-u_{xx}= \tilde u_x \tilde{\alpha}_t-\tilde u_{t} \tilde{\alpha}_x-(\tilde u_x \tilde{\alpha}_x)_x+f(\tilde{u})\tilde{\alpha}_x,
\end{equation}
where it is understood that the argument of the function $\tilde u$ and its derivatives appearing
on the righthand side are evaluated at $(x+\tilde{\alpha}(x,t),t)$.  Moreover, by another direct calculation,
using the fact that
$$
L(\bar{u}^{\prime}(x))=\left(\partial_x^2+Df(\bar{u})\right)\bar{u}^{\prime}(x)=0,
$$
by translation invariance, we have
\begin{equation*}
\left(\partial_t-L\right)\bar{u}'(x)\tilde{\alpha}=\bar{u}_x\tilde{\alpha}_t
-(\bar{u}_x\tilde{\alpha}_{x})_{x}-\bar{u}_{xx}\tilde{\alpha}_x
=\bar{u}_x\tilde{\alpha}_t -(\bar{u}_x\tilde{\alpha}_{x})_{x}+f(\bar{u})\tilde{\alpha}_x.
\end{equation*}
Subtracting, and using the facts that,
by differentiation of $(\bar u+ v)(x,t)= \tilde u(x+\tilde{\alpha},t)$,
\begin{equation}\label{keyderivs}
\begin{aligned}
\bar u_x + v_x&= \tilde u_x(1+\tilde{\alpha}_x),\\
\bar u_t + v_t&= \tilde u_t + \tilde u_x\tilde{\alpha}_t,\\
\end{aligned}
\end{equation}
so that
\begin{equation}\label{solvedderivs}
\begin{aligned}
\tilde u_x-\bar u_x -v_x&=
-(\bar u_x+v_x) \frac{\tilde{\alpha}_x}{1+\tilde{\alpha}_x},\\
\tilde u_t-\bar u_t -v_t&=
-(\bar u_x+v_x) \frac{\tilde{\alpha}_t}{1+\tilde{\alpha}_x},\\
\end{aligned}
\end{equation}
we obtain
\begin{align*}
u_t-f(u)-u_{xx}&=
(\partial_t-L)\bar{u}'(x)\tilde{\alpha}
+v_x\tilde{\alpha}_t - v_t \tilde{\alpha}_x - (v_x\tilde{\alpha}_x)_x\\
&\quad+\left((\bar u_x +v_x)\frac{\tilde{\alpha}_x^2}{1+\tilde{\alpha}_x} \right)_x
 +\left(f(v+\bar{u})-f(\bar{u})\right)\tilde{\alpha}_x,
\end{align*}
yielding \eqref{eqn:1nlper} by
$v_x\tilde{\alpha}_t - v_t \tilde{\alpha}_x = (v\tilde{\alpha}_t)_x-(v\tilde{\alpha}_x)_t$
and
$(v_x\tilde{\alpha}_x)_x= (v\tilde{\alpha}_x)_{xx} - (v\tilde{\alpha}_{xx})_{x} $.
\end{proof}

\begin{cor}\label{cor:canest}
The nonlinear residual $v$ defined in \eqref{pertvar} satisfies
\begin{equation}\label{veq}
\left(\partial_t-L\right)v=\left(\partial_t-L\right)\bar{u}'(x)\tilde{\alpha}
+Q+ R_x +(\partial_x^2+\partial_t)S+T,
\end{equation}
where
\begin{equation}\label{eqn:Q}
Q:=f(v(x,t)+\bar{u}(x))-f(\bar{u}(x))-Df(\bar{u}(x))v=\mathcal{O}(|v|^2),
\end{equation}
\begin{equation}\label{eqn:R}
R:= v\tilde{\alpha}_t + v\tilde{\alpha}_{xx}+  (\bar u_x +v_x)\frac{\tilde{\alpha}_x^2}{1+\tilde{\alpha}_x},
\end{equation}
\begin{equation}\label{eqn:S}
S:=-v\tilde{\alpha}_x =\mathcal{O}(|v| |\tilde{\alpha}_x|),
\end{equation}
and
\begin{equation}\label{eqn:T}
T:=\left(f(v+\bar{u})-f(\bar{u})\right)\tilde{\alpha}_x=\mathcal{O}(|v||\tilde{\alpha}_x|).
\end{equation}
\end{cor}

\begin{proof}
Straightforward Taylor expansion comparing \eqref{eqn:1nlper} and
$\bar u_t - f(\bar u)-\bar u_{xx}=0$.
\end{proof}

Using Corollary \ref{cor:canest} and applying Duhamel's principle, 
taking $\tilde \alpha(\cdot, 0)=0$ similarly as before,
we obtain the integral (implicit) representation
\begin{align*}
v(x,t)=&\bar{u}'(x)\tilde{\alpha}(x,t)+\int_{-\infty}^\infty G(x,t;y)v_0(y)dy\\
&+\int_0^t\int_{-\infty}^\infty G(x,t-s;y)\left(Q+R_y+(\partial_y^2+\partial_s)S+T\right)(y,s)dyds
\end{align*}
for the nonlinear perturbation $v$.  Thus, if we define $\tilde{\alpha}$ implicitly via the formula
\begin{equation} \label{eqn:tildealpha}
\begin{aligned}
\tilde{\alpha}(x,t):=&-\int_{-\infty}^\infty \tilde{e}(x,t;y)v_0(y)dy\\
&-\int_0^t\int_{-\infty}^\infty \tilde{e}(x,t-s;y)\left(Q+R_y+(\partial_y^2+\partial_s)S+T\right)(y,s)dyds,
\end{aligned}
\end{equation}
we obtain the integral representation
\begin{equation} \label{eqn:vint}
\begin{aligned}
v(x,t)=&\int_{-\infty}^{+\infty} \widetilde{H}(x,t;y)v_0(y)dy\\
&+\int_0^t\int_{-\infty}^{+\infty} \widetilde{H}(x,t-s;y)\left(Q+R_y+(\partial_y^2+\partial_s)S+T\right)(y,s)dyds.
\end{aligned}
\end{equation}
Moreover, differentiating and recalling that $\tilde{e}(x,t;y)=0$ for $0<t\leq 1$ we obtain
\begin{equation} \label{eqn:psiint}
\begin{aligned}
&\partial_t^k\partial_x^m\tilde{\alpha}(x,t)\\
:=&-\int_{-\infty}^{+\infty} \partial_t^k\partial_x^m\tilde{e}(x,t;y)v_0(y)dy\\
&-\int_0^t\int_{-\infty}^{+\infty} \partial_t^k\partial_x^m \tilde{e}(x,t-s;y)\left(Q+R_y+(\partial_y^2+\partial_s)S+T\right)(y,s)dyds.
\end{aligned}
\end{equation}
Together, \eqref{eqn:vint}-\eqref{eqn:psiint} form a complete system in the variables $\left(v,\partial_t^k\tilde{\alpha},\partial_x^m\tilde{\alpha}\right)$,
$0\leq k\leq 1$, $0\leq m\leq K+1$, $k+m\geq 1$, where $K\geq 2$ is a constant.  
(Note, again, that \eqref{eqn:tildealpha} gives $\tilde \alpha(\cdot, 0)=0$, justifying our derivation.)
Given a solution of system \eqref{eqn:vint}-\eqref{eqn:psiint}, we may 
recover the shift function $\tilde{\alpha}$ 
by integrating $\tilde \alpha_x$ with respect to $x$ and using $\lim_{x\to \pm \infty}\tilde \alpha(x,t)=0$.

Now, from the original differential equation \eqref{veq} together with \eqref{eqn:psiint}, we readily obtain short-time existence and continuity with respect to $t$ of solution $(v,\tilde{\alpha}_t,\tilde{\alpha}_x)\in H^{K}$ by a standard contraction-mapping argument treating the linear $Df(\bar u)v$ term of the left-hand side along with
$Q,R,S,T,\tilde{\alpha}\bar u'$
terms of the right-hand side as sources in the heat equation.

{\bf Notation.} The Sobolev space $H^K(\mathbb{R})$ is defined as
$$H^K(\mathbb{R})=\{u \in L^2(\mathbb{R}): D^{\beta}u \in L^2(\mathbb{R}),\forall |\beta|\leq K\}$$
equipped with the norm
$$\|u\|_{H^K(\mathbb{R})}:=\left(\sum_{|\beta|\leq K}\|D^{\beta}u\|_{L^2(\mathbb{R})}^2\right)^{\frac{1}{2}}.$$

\section{$H^K$ Nonlinear Iteration} \label{HKNonlinearIteration}

Associated with the solution $(u,\tilde{\alpha}_t,\tilde{\alpha}_x)$ of the integral system \eqref{eqn:vint}-\eqref{eqn:psiint}, we define
\begin{equation}\label{eqn:eta}
\begin{aligned}
\zeta_1(t):=&\sup_{0\leq s\leq t}\left(\|v\|_{H^{K}(x;\mathbb{R})}(s)e^{\eta_0s}
+\|(\tilde{\alpha}_t,\tilde{\alpha}_x)\|_{H^{K}(x;\mathbb{R})}(s)(1+s)^{\frac{3}{4}}\right).
\end{aligned}
\end{equation}
By short time $H^K(\mathbb{R})$ existence theory, the quantities $\|v\|_{H^K(\mathbb{R})}$ and \\ $\|(\tilde{\alpha}_t,\tilde{\alpha}_x)\|_{H^K(\mathbb{R})}$ are continuous so long as they remain small.  Thus, $\zeta_1$ is a continuous function of $t$ as long as it remains small. We now use the linearized Green function estimates of Section \ref{LinearGreenEst} to prove that if $\zeta_1$ is initially small then it must remain so.

\begin{lem}\label{lem:eta}
For all $t\geq 0$ for which $\zeta_1(t)$ is finite, we have the estimate
$$
\zeta_1(t)\leq C\left( E_0+\zeta_1(t)^2\right)
$$
for some constant $C>0$, so long as $E_0:=\|v(\cdot,0)\|_{L^1(\mathbb{R})\cap H^{K}(\mathbb{R})}$ is sufficiently small.
\end{lem}

\begin{proof}
To begin, notice that by the descriptions of $Q$, $T$, $R$, and $S$ in Corollary \ref{cor:canest} we have that
\begin{align*}
\|Q(\cdot,t)\|_{L^1(\mathbb{R})}&\leq \|v\|_{H^K(x;\mathbb{R})}^2 \leq C \zeta_1(t)^2e^{-2\eta_0t}\\
\|R_y(\cdot,t)\|_{L^1(\mathbb{R})}&\leq \|v\|_{H^K(x;\mathbb{R})}\|(\tilde{\alpha}_t,\tilde{\alpha}_x)\|_{H^{K+1}(x;\mathbb{R})} \leq C \zeta_1(t)^2e^{-\eta_0t}(1+t)^{-\frac{3}{2}}\\
\|T(\cdot,t)\|_{L^1(\mathbb{R})}&\leq \|v\|_{H^K(x;\mathbb{R})}\|(\tilde{\alpha}_t,\tilde{\alpha}_x)\|_{H^{K+1}(x;\mathbb{R})} \leq C \zeta_1(t)^2e^{-\eta_0t}(1+t)^{-\frac{3}{2}}\\
\|(\partial_t+\partial_x^2)S(\cdot,t)\|_{L^1(\mathbb{R})}&\leq \|v\|_{H^K(x;\mathbb{R})}\|(\tilde{\alpha}_t,\tilde{\alpha}_x)\|_{H^{K+1}(x;\mathbb{R})} \leq C \zeta_1(t)^2e^{-\eta_0t}(1+t)^{-\frac{3}{2}}\\
\end{align*}
so long as $\|(v_x,\tilde{\alpha}_x)(\cdot,t)\|_{L^\infty(\mathbb{R})}\leq \|(v,\tilde{\alpha}_x)\|_{H^{K}(x;\mathbb{R})}(t)\leq \zeta_1(t)$ remains bounded.

Thus, applying the bounds \eqref{10.1} and \eqref{10.5} of Proposition \ref{Pro10.2} to representations \eqref{eqn:vint}-\eqref{eqn:psiint}, we obtain for any
$2\leq p\leq\infty$ the bound
\begin{equation}\label{vbds}
\begin{aligned}
	\|v(\cdot,t)\|_{L^p(\mathbb{R})}&\leq C e^{-\eta_0 t}E_0\\
&~~~~~+ C\zeta_1^2(t)\int_0^t e^{-\eta_0(t-s)}
	\left(e^{-2\eta_0s}+e^{-\eta_0s}(1+s)^{-\frac{3}{2}}\right)ds  \\
&\leq C\left(E_0+\zeta_1(t)^2\right)e^{-\eta_0t},
\end{aligned}
\end{equation}
and similarly using \eqref{10.5} we have
\begin{equation}\label{psibds}
\begin{aligned}
&\|(\tilde{\alpha}_t,\tilde{\alpha}_x)(\cdot,t)\|_{W^{K+1,p}(x;\mathbb{R})}\\
\leq &
C (1+t)^{-\frac{1}{2}\left(1-\frac{1}{p}\right)-\frac{1}{2}}E_0\\
&+ C \zeta_1(t)^2\int_0^t(1+t-s)^{-\frac{1}{2}\left(1-\frac{1}{p}\right)-\frac{1}{2}}
\left(e^{-2\eta_0s}+e^{-\eta_0s}(1+s)^{-\frac{3}{2}}\right)ds  \\
\leq & C\left(E_0+\zeta_1(t)^2\right)(1+t)^{-\frac{1}{2}\left(1-\frac{1}{p}\right)-\frac{1}{2}},
\end{aligned}
\end{equation}
yielding in particular that $\|(\tilde{\alpha}_t,\tilde{\alpha}_x)\|_{H^{K+1}}$ is arbitrarily small if $E_0$ and $\zeta_1(t)$ are,
thus verifying the hypothesis of Proposition \ref{prop:nd} below.
By the nonlinear damping estimate given in Proposition \ref{prop:nd}, therefore, the size of $v$ in $H^K(\mathbb{R})$ can be controlled by its size in $L^2(\mathbb{R})$ together with $H^K$ estimates on the derivatives of the phase function $\tilde{\alpha}$. In particular, we have for some positive constants $\theta_1$ and $\theta_2$
\begin{align*}
&\|v(\cdot,t)\|_{H^K(\mathbb{R})}^2 \leq \\
& C  e^{-\theta_1 t}E_0^2+C\left(E_0+ C\zeta_1(t)^2\right)^2\int_0^t e^{-\theta_2(t-s)}\left(e^{-2\eta_0s}+(1+s)^{-\frac{3}{4}}\right)ds  \\
\leq& C e^{-\theta_1 t}E_0^2+C\left(E_0+\zeta_1(t)^2\right)^2e^{-2\eta_0t}\\
\leq& C \left(E_0+\zeta_1(t)^2\right)^2e^{-2\eta_0t}.
\end{align*}
This estimate together with \eqref{psibds} in the case $p=2$ completes the proof.
\end{proof}

\begin{pro}\label{prop:nd}
Assuming $(\mathcal{D})$, let $v(\cdot,0)\in H^K(\mathbb{R})$ (for $v$ as in \eqref{pertvar}) and suppose that for some $T>0$ the $H^K(\mathbb{R})$
norm of $v$ and the $H^{K+1}(\mathbb{R})$ norms of $\tilde{\alpha}_t(\cdot,t)$ and $\tilde{\alpha}_x(\cdot,t)$ remain
bounded by a sufficiently small constant for all $0\leq t\leq T$.  Then there are constants
$\theta_1, \theta_2, C>0$ such that
\begin{equation*}
\begin{aligned}
	\|v(\cdot,t)\|^2_{H^K(\mathbb{R})}\leq & C e^{-\theta_1 t}\|v(\cdot,0)\|_{H^K(\mathbb{R})}^2\\
					&+ C\int_0^te^{-\theta_2(t-s)}\left(\|v(\cdot,s)\|^2_{L^2(\mathbb{R})}+\|(\tilde{\alpha}_t, \tilde{\alpha}_x)(\cdot,s)\|_{H^K(\mathbb{R})}^2\right)ds .
\end{aligned}
\end{equation*}
for all $0\leq t\leq T$.
\end{pro}

\begin{proof}
Subtracting from the equation \eqref{altform} for $u$
the equation for $\bar u$, we may write the
nonlinear perturbation equation as
\begin{equation}\label{vperturteq}
v_t - Df(\bar u)v-v_{xx}= Q
+ \tilde u_x \tilde{\alpha}_t -\tilde u_{t} \tilde{\alpha}_x - (\tilde u_x \tilde{\alpha}_x)_x+f(\tilde{u})\psi_x,
\end{equation}
where it is understood that derivatives of $\tilde u$ appearing
on the right-hand side
are evaluated at $(x+\tilde{\alpha}(x,t),t)$.
Using \eqref{solvedderivs} to replace $\tilde u_x$ and
$\tilde u_t$ respectively by
$\bar u_x + v_x -(\bar u_x+v_x) \frac{\tilde{\alpha}_x}{1+\tilde{\alpha}_x}$
and
$\bar u_t + v_t -(\bar u_x+v_x) \frac{\tilde{\alpha}_t}{1+\tilde{\alpha}_x}$,
and moving the resulting $v_t\tilde{\alpha}_x$ term to the left-hand side
of \eqref{vperturteq}, we obtain
\begin{equation}\label{vperturteq2}
\begin{aligned}
(1+\tilde{\alpha}_x) v_t -v_{xx}&=
-Df(\bar u)v+ Q
+ (\bar u_x+v_x) \tilde{\alpha}_t
\\ &\quad
- ((\bar u_x+v_x)  \tilde{\alpha}_x)_x
+ \Big((\bar u_x+v_x) \frac{\tilde{\alpha}_x^2}{1+\tilde{\alpha}_x}\Big)_x+f(\tilde{u})\tilde{\alpha}_x
\end{aligned}
\end{equation}
Taking the $L^2$ inner product in $x$ of
$\sum_{j=0}^K \frac{(-1)^{j}\partial_x^{2j}v}{1+\tilde{\alpha}_x}$
against \eqref{vperturteq2}, integrating by parts,
and rearranging the resulting terms,
we arrive at the inequality
\begin{equation*}
\begin{aligned}
&\partial_t \|v(\cdot,t)\|_{H^K(\mathbb{R})}^2 \\
\leq& -\theta \|\partial_x^{K+1} v(\cdot,t)\|_{L^2(\mathbb{R})}^2 +
C\left( \|v(\cdot,t)\|_{H^K(\mathbb{R})}^2
+\|(\tilde{\alpha}_t, \tilde{\alpha}_x)(\cdot,s)\|_{H^K(\mathbb{R})}^2 \right),
\end{aligned}
\end{equation*}
for some $\theta>0$, $C>0$, so long as $\|\tilde u\|_{H^K(\mathbb{R})}$ remains bounded,
and $\|v(\cdot,t)\|_{H^K(\mathbb{R})}$ and $\|(\tilde{\alpha}_t, \tilde{\alpha}_x)(\cdot,t)\|_{H^{K+1}(\mathbb{R})}$ remain sufficiently small.
Using the Sobolev interpolation
$
\|g\|_{H^K(\mathbb{R})}^2 \leq  \tilde{C}^{-1}\|\partial_x^{K+1} g\|_{L^2(\mathbb{R})}^2 + \tilde{C} \| g\|_{L^2(\mathbb{R})}^2
$
for $\tilde{C}>0$ sufficiently large, we obtain
\begin{equation*}
\begin{aligned}
&\partial_t \|v(\cdot,t)\|_{H^K(\mathbb{R})}^2(t)\\
\leq& -\tilde{\theta} \|v(\cdot,t)\|_{H^K(\mathbb{R})}^2 +
C\left( \|v(\cdot,t)\|_{L^2(\mathbb{R})}^2+\|(\tilde{\alpha}_t, \tilde{\alpha}_x)(\cdot,s)\|_{H^K(\mathbb{R})}^2 \right)
\end{aligned}
\end{equation*}
from which the desired estimate follows by Gronwall's inequality.
\end{proof}

\section{Pointwise Nonlinear Iteration and Pointwise Bound on the Perturbation} \label{PointwiseNonlinearIterationandPointwiseBoundonthePerturbation}

In this section, we give a proof of the Theorem \ref{mainthm2} using the improved pointwise bounds stated in Proposition \ref{Proposition1.3}. Associated with the solution $(u,\tilde{\alpha})$ of the integral system \eqref{eqn:vint} and \eqref{eqn:psiint}, we define
\begin{equation} \label{zeta2}
\begin{aligned}
\zeta_2(t):=\sup_{0\leq s\leq t,y \in \mathbb{R}}
&\left((|v(y,s)|+|v_x(y,s)|+|v_{xx}(y,s)|)(1+s)^{\frac{1}{2}}e^{\frac{\eta_0}{2}s+\frac{|y|^2}{2Ms}}\right.\\
&\left.+\|\tilde{\alpha}\|_{W^{3,\infty}(x;\mathbb{R})}(s)
+\|(\tilde{\alpha}_t,\tilde{\alpha}_x)\|_{W^{3,\infty}(x;\mathbb{R})}(s)(1+s)^{\frac{1}{2}}\right).
\end{aligned}
\end{equation}
\begin{lem}
For all $t\geq 0$ for which $\zeta_2(t)$ is finite, we have the estimate
\begin{equation}
\zeta_2(t)\leq C(E_0+\zeta_2(t)^2) \label{zeta_est2}
\end{equation}
for some constant $C>0$, so long as $E_0:=\|v(\cdot,0)\|_{L^1(\mathbb{R})\cap H^{K}(\mathbb{R})}>0$ is sufficiently small.
\end{lem}
\begin{proof}
Let us recall the definition of the Gaussian probability density function $$K(x,t)=(2\pi t)^{-\frac{1}{2}}e^{-\frac{x^2}{2t}},$$ and the semigroup property $K(\cdot,t_1)\ast K(\cdot,t_2)=K(\cdot,t_1+t_2)$.

If we define $$K_M(x,t)=t^{-\frac{1}{2}}e^{-\frac{x^2}{Mt}},$$ we represent it in terms of $K$ as $$K_M(x,t)=\sqrt{\pi M}\left(2\pi \left(\frac{Mt}{2}\right)\right)^{-\frac{1}{2}}e^{-\frac{|x|^2}{2\left(\frac{Mt}{2}\right)}}=\sqrt{\pi M}\cdot K(x,\frac{Mt}{2}).$$ The semigroup property becomes $K_M(\cdot,t_1)\ast K_M(\cdot,t_2)=\sqrt{\pi M}\cdot K_M(\cdot,t_1+t_2)$.

Using the representations on $Q,R,S$ and $T$, we can conclude that
\begin{eqnarray*}
|Q(y,s)|&\leq&C\zeta_2^2(t)(1+s)^{-1}e^{-\eta_0s-\frac{|y|^2}{Ms}}\leq C\zeta_2^2(t)(1+s)^{-\frac{1}{2}}s^{-\frac{1}{2}}e^{-\eta_0s-\frac{|y|^2}{Ms}},\\
|R_y(y,s)|&\leq&C\zeta_2^2(t)(1+s)^{-\frac{1}{2}-\frac{1}{2}}e^{-\frac{\eta_0}{2}s-\frac{|y|^2}{2Ms}}\leq C\zeta_2^2(t)(1+s)^{-\frac{1}{2}}s^{-\frac{1}{2}}e^{-\frac{\eta_0}{2}s-\frac{|y|^2}{2Ms}},\\
|(\partial_{y}^{2}+\partial_{s})S(y,s)|&\leq&C\zeta_2^2(t)(1+s)^{-\frac{1}{2}-\frac{1}{2}}
e^{-\frac{\eta_0}{2}s-\frac{|y|^2}{2Ms}}\leq C\zeta_2^2(t)(1+s)^{-\frac{1}{2}}s^{-\frac{1}{2}}e^{-\frac{\eta_0}{2}s-\frac{|y|^2}{2Ms}},\\
|T(y,s)|&\leq&C\zeta_2^2(t)(1+s)^{-\frac{1}{2}-\frac{1}{2}}e^{-\frac{\eta_0}{2}s-\frac{|y|^2}{2Ms}}\leq C\zeta_2^2(t)(1+s)^{-\frac{1}{2}}s^{-\frac{1}{2}}e^{-\frac{\eta_0}{2}s-\frac{|y|^2}{2Ms}},
\end{eqnarray*}
for $0<s\leq t$.

From \eqref{eqn:vint}, using the pointwise bound \eqref{tilde_H_estimate} on $\tilde{H}$ and the four estimates above we can derive that
\begin{eqnarray*}
& &   |v(x,t)|\\
&\leq&CE_0\int_{-\infty}^{+\infty}e^{-\eta_0 t}\frac{e^{-\frac{|x-y|^2}{Mt}}}{\sqrt{t}}\cdot
      \frac{e^{-\frac{|y|^2}{M}}}{\sqrt{1}}dy\\
&    &+C\zeta_2^2(t)\int_0^t\int_{-\infty}^{+\infty}(1+s)^{-\frac{1}{2}}
      e^{-\eta_0 (t-s)}\frac{e^{-\frac{|x-y|^2}{M(t-s)}}}{\sqrt{t-s}}\cdot
      e^{-\eta_0 s}\frac{e^{-\frac{|y|^2}{Ms}}}{\sqrt{s}}dyds\\
&    &+C\zeta_2^2(t)\int_0^t\int_{-\infty}^{+\infty}(1+s)^{-\frac{1}{2}}
      e^{-\eta_0 (t-s)}\frac{e^{-\frac{|x-y|^2}{M(t-s)}}}{\sqrt{t-s}}\cdot
      e^{-\frac{\eta_0}{2} s}\frac{e^{-\frac{|y|^2}{2Ms}}}{\sqrt{s}}dyds\\
&\leq&CE_0\int_{-\infty}^{+\infty}e^{-\eta_0t}K_M(x-y,t)\cdot K_M(y,1)dy\\
&    &+C\zeta_2^2(t)\int_0^t\int_{-\infty}^{+\infty}e^{-\eta_0t}(1+s)^{-\frac{1}{2}}K_M(x-y,t-s)\cdot K_M(y,s)dyds\\
&    &+C\zeta_2^2(t)\int_0^t\int_{-\infty}^{+\infty}e^{-\eta_0t+\frac{\eta_0}{2}s}(1+s)^{-\frac{1}{2}}
      K_{2M}(x-y,t-s)\cdot K_{2M}(y,s)dyds\\
&\leq&C\sqrt{\pi M}E_0e^{-\eta_0t}K_M(x,t+1)\\
&    &+C\sqrt{\pi M}\zeta_2^2(t)e^{-\eta_0t}\int_0^t(1+s)^{-\frac{1}{2}}K_M(x,t)ds\\
&    &+C\sqrt{2\pi M}\zeta_2^2(t)e^{-\eta_0t}\int_0^t e^{\frac{\eta_0}{2}s}(1+s)^{-\frac{1}{2}}K_{2M}(x,t)ds\\
&\leq&CE_0e^{-\eta_0t}K_M(x,t+1)+C\zeta_2^2(t)2\sqrt{t}e^{-\eta_0t}K_M(x,t)\\
&    &+C\zeta_2^2(t)e^{-\frac{\eta_0}{2}t}K_{2M}(x,t)\\
&\leq&C(E_0+\zeta_2^2(t))(1+t)^{-\frac{1}{2}}e^{-\frac{\eta_0}{2}t-\frac{|x|^2}{2Mt}},
\end{eqnarray*}
and using the pointwise bound \eqref{tilde_e(y,t)estimate} on $\partial_t^k\partial_x^m\tilde{e}(x,t;y)$,
\begin{eqnarray*}
& &   |\partial_t^k\partial_x^m\tilde{\alpha}(x,t)|\\
&\leq&CE_0\int_{-\infty}^{+\infty}(1+t)^{-\frac{\tau(m+k)}{2}}e^{-\eta_0|y|}e^{-\frac{|y|^2}{M}}dy\\
&    &+C\zeta_2^2(t)\int_0^t\int_{-\infty}^{+\infty}(1+t-s)^{-\frac{\tau(m+k)}{2}}e^{-\eta_0|y|}
      e^{-\eta_0 s}(1+s)^{-\frac{1}{2}}\frac{e^{-\frac{|y|^2}{Ms}}}{\sqrt{s}}dyds\\
&    &+C\zeta_2^2(t)\int_0^t\int_{-\infty}^{+\infty}(1+t-s)^{-\frac{\tau(m+k)}{2}}e^{-\eta_0|y|}
      e^{-\frac{\eta_0}{2} s}(1+s)^{-\frac{1}{2}}\frac{e^{-\frac{|y|^2}{2Ms}}}{\sqrt{s}}dyds\\
&\leq&CE_0(1+t)^{-\frac{\tau(m+k)}{2}}\\
&    &+C\zeta_2^2(t)\int_{0}^{t}(1+t-s)^{-\frac{\tau(m+k)}{2}}(1+s)^{-\frac{1}{2}}e^{-\eta_0 s}
      \left(\int_{-\infty}^{+\infty}e^{-\eta_0|y|}\frac{e^{-\frac{|y|^2}{Ms}}}{\sqrt{s}}dy\right)ds\\
&    &+C\zeta_2^2(t)\int_{0}^{t}(1+t-s)^{-\frac{\tau(m+k)}{2}}(1+s)^{-\frac{1}{2}}e^{-\frac{\eta_0}{2} s}
      \left(\int_{-\infty}^{+\infty}e^{-\eta_0|y|}\frac{e^{-\frac{|y|^2}{2Ms}}}{\sqrt{s}}dy\right)ds\\
&\leq&\left\{
 \begin{array}{l l}
 C(E_0+\zeta_2^2(t))(1+t)^{-\frac{1}{2}} \quad \text{for $m+k\geq 1$;}\\
 C(E_0+\zeta_2^2(t)) \quad\quad\quad\quad\quad \text{for $m+k=0$.}
 \end{array}
 \right.
\end{eqnarray*}
where $\tau(m+k)=1$ for $m+k \geq 1$ and $0$ otherwise. Thus the Lemma follows.
\end{proof}
Finally, we give the proof of Theorem \ref{mainthm2}.
\begin{proof}[Proof of Theorem \ref{mainthm2}]
By continuous induction, we have that $\zeta_2(t)\leq 2CE_0$. Hence the stated estimate on
$v(x,t)=\tilde{u}(x+\tilde{\alpha}(x,t),t)-\bar{u}(x)$ follows. Now we prove the other two bounds on $\tilde{\alpha}$ and $\partial_t^k\partial_x^m\tilde{\alpha}$. From the formula \eqref{eqn:psiint} for $\partial_t^k\partial_x^m\tilde{\alpha}$ and the bounds on $Q, R_y, (\partial_y^2+\partial_s)S, T$, the pointwise bound \eqref{tilde_e(y,t)estimate} on $\partial_t^k\partial_x^m\tilde{e}(x,t;y)$, we obtain 
	for $t\geq 1$, $k+m\geq 1$ ($\tilde{\alpha}(x,t)=0$ when $0<t<1$),
\begin{align*} \label{drevtildealphapointwise}
&    |\partial_t^k\partial_x^m\tilde{\alpha}(x,t)|\\
\leq&C E_0 \int_{-\infty}^{+\infty}\left|\frac{e^{-\frac{(x-y+t)^2}{Mt}}}{\sqrt{t}}-\frac{e^{-\frac{(x-y-t)^2}{Mt}}}{\sqrt{t}}\right|
e^{-\eta|y|}e^{-\frac{|y|^2}{M}}dy\\
&+C E_0 \int_1^t\int_{-\infty}^{+\infty}\left|\frac{e^{-\frac{(x-y+(t-s))^2}{M(t-s)}}}{\sqrt{t-s}}
-\frac{e^{-\frac{(x-y-(t-s))^2}{M(t-s)}}}{\sqrt{t-s}}\right|
	s^{-1/2}
	e^{-\eta|y|-\eta_0s-\frac{|y|^2}{Ms}}dyds\\
\leq&C E_0 \int_{-\infty}^{+\infty}\left(K_M(x+t-y,t)+K_M(x-t-y,t)\right)
K_M(y,1)e^{-\eta|y|}dy\\
&+C E_0 \int_1^t\int_{-\infty}^{+\infty}K_M(x+(t-s)-y,t-s)K_M(y,s)e^{-\eta|y|-\eta_0s}dyds\\
&+C E_0 \int_1^t\int_{-\infty}^{+\infty}K_M(x-(t-s)-y,t-s)K_M(y,s)e^{-\eta|y|-\eta_0s}dyds\\
\leq&C E_0 \left(K_M(x+t,t+1)+K_M(x-t,t+1)\right)\\
&+C E_0 \int_1^tK_M(x+t,t)e^{-\eta_0s}ds+C E_0 \int_1^tK_M(x-t,t)e^{-\eta_0s}ds\\
	\leq&CE_0 t^{-1/2} \left(e^{-\frac{|x+t|^2}{Mt}}+e^{-\frac{|x-t|^2}{Mt}}\right)
\leq CE_0 \left(e^{-\frac{|x+t|^2}{Mt}}+e^{-\frac{|x-t|^2}{Mt}}\right).
\end{align*}
The bound on $\tilde{\alpha}$ can be obtained by integrating the bound on $\tilde{\alpha}_x(x,t)$ from $x$ to $\pm\infty$, using 
	$\lim_{x\to\pm\infty}\tilde{\alpha}(x,t)=0$.
This completes the proof of Theorem \ref{mainthm2}.
\end{proof}

\subsection*{\bf{Acknowledgements}}
This project was completed while studying within the PhD program at Indiana University, Bloomington. Thanks to my thesis advisor Kevin Zumbrun for suggesting the problem and for helpful discussions. 
Thanks also to the referee for his careful reading and many helpful suggestions.

\subsection*{\bf{Disclosures}}
Conflict of Interest: The author declares that he has no conflict of interest.

\bibliographystyle{spmpsci}      


\bibliographystyle{amsalpha}
\bibliography{books}

\end{document}